\documentclass[12pt]{amsart}
\usepackage{txfonts}
\usepackage{amscd,amssymb,amsmath}
 \usepackage[foot]{amsaddr}
\usepackage[pdftex,colorlinks,plainpages,backref,urlcolor=blue]{hyperref}
\usepackage{url}
\usepackage{array}
\usepackage{enumitem}
\usepackage{tikz}
\usetikzlibrary{cd}
\usepackage{spectralsequences}

\newtheorem{theorem}{Theorem}[section]
\newtheorem{lemma}[theorem]{Lemma}
\newtheorem{corollary}[theorem]{Corollary}
\newtheorem{proposition}[theorem]{Proposition}

\theoremstyle{definition}

\newtheorem{example}[theorem]{Example}
\newtheorem{remark}[theorem]{Remark}

\DeclareMathOperator{\codim}{codim}

\DeclareMathOperator{\id}{id}
\DeclareMathOperator{\rank}{rank}

\DeclareMathOperator{\Hom}{Hom}
\DeclareMathOperator{\gr}{gr}
\DeclareMathOperator{\im}{im}

\DeclareMathOperator{\ab}{ab}

\DeclareMathOperator{\Lie}{Lie}

\newcommand{\mc}[1]{\mathcal{#1}}

\DeclareMathSymbol{\Gamma}{\mathalpha}{operators}{0}
\newcommand{\G}{\Gamma}

\newcommand{\Z}{\mathbb{Z}}
\newcommand{\Q}{\mathbb{Q}}
\newcommand{\R}{\mathbb{R}}
\newcommand{\C}{\mathbb{C}}

\newcommand{\LL}{\mathbf{L}}
\renewcommand{\k}{\Bbbk}
\newcommand{\A}{\mathcal{A}}
\newcommand{\B}{\mathcal{B}}

\newcommand{\mf}[1]{\mathfrak{#1}}
\newcommand{\ov}[1]{\overline{#1}}

\newcommand{\Gg}[1]{G/\G_{#1}(G)}

\newcommand{\mcf}{\mc{F}}

\newcommand{\h}{\mf{h}}

\newcommand{\GG}[1]{G/\Gamma_{#1}(G)}
\newcommand{\GGG}[2]{{#1}/\G_{#2}({#1})}

\newcommand{\PL}{\scriptscriptstyle{\rm PL}}

\newcommand{\surj}{\twoheadrightarrow}
\newcommand{\inj}{\hookrightarrow}

\newcommand{\isom}{\xrightarrow{
   \,\smash{\raisebox{-0.65ex}{\ensuremath{\scriptstyle\simeq}}}\,}}

\makeatletter
\newsavebox{\@brx}
\newcommand{\llangle}[1][]{\savebox{\@brx}{\(\m@th{#1\langle}\)}%
  \mathopen{\copy\@brx\kern-0.5\wd\@brx\usebox{\@brx}}}
\newcommand{\rrangle}[1][]{\savebox{\@brx}{\(\m@th{#1\rangle}\)}%
  \mathclose{\copy\@brx\kern-0.5\wd\@brx\usebox{\@brx}}}
\makeatother

\newcommand{\arxiv}[1]
{\texttt{\href{http://arxiv.org/abs/#1}{arXiv:#1}}}

\newcommand{\abs}[1]{\left| #1 \right|}

\numberwithin{table}{section}
\numberwithin{equation}{section}
\def\norms#1#2{\left\| #1 \right\|_{\lower 1ex\hbox{$\scriptstyle #2 $}}}

\makeatletter
\def\namedlabel#1#2{\begingroup
    #2%
    \def\@currentlabel{#2}%
    \phantomsection\label{#1}.\endgroup
}
\makeatother
\makeatletter
\def\@tocline#1#2#3#4#5#6#7{\relax
  \ifnum #1>\c@tocdepth % then omit
  \else
    \par \addpenalty\@secpenalty\addvspace{#2}%
    \begingroup \hyphenpenalty\@M
    \@ifempty{#4}{%
      \@tempdima\csname r@tocindent\number#1\endcsname\relax
    }{%
      \@tempdima#4\relax
    }%
    \parindent\z@ \leftskip#3\relax \advance\leftskip\@tempdima\relax
    \rightskip\@pnumwidth plus4em \parfillskip-\@pnumwidth
    #5\leavevmode\hskip-\@tempdima
      \ifcase #1
       \or\or \hskip 1em \or \hskip 2em \else \hskip 3em \fi%
      #6\nobreak\relax
    \dotfill\hbox to\@pnumwidth{\@tocpagenum{#7}}\par
    \nobreak
    \endgroup
  \fi}
\makeatother
 
\title[Homology, lower central series, and hyperplane arrangements]%
{Homology, lower central series, and \\ hyperplane arrangements}

\dedicatory{To the memory of \c{S}tefan Papadima, 1953--2018}

\author{Richard~D.~Porter$^1$}
\author{Alexander~I.~Suciu$^{1,2}$}
\address{$^1$Department of Mathematics,
Northeastern University,
Boston, MA 02115, USA}
\email{\href{mailto:r.porter@northeastern.edu}{r.porter@northeastern.edu}}

\email{\href{mailto:a.suciu@northeastern.edu}{a.suciu@northeastern.edu}}
\thanks{$^2$Supported in part by the Simons Foundation Collaboration Grant 
for Mathematicians \#354156}



\begin{document}
 
 \begin{abstract}
We explore finitely generated groups by studying the nilpotent 
towers and the various Lie algebras attached to such groups. 
Our main goal is to relate an isomorphism extension problem 
in the Postnikov tower to the existence of certain commuting diagrams. 
This recasts a result of G.~Rybnikov in a more general framework
and leads to an application to hyperplane arrangements, whereby 
we show that all the nilpotent quotients of a decomposable arrangement 
group are combinatorially determined.
\end{abstract}
 
\maketitle

\tableofcontents

\section{Introduction}
\label{sect:intro}

\subsection{Motivation}
\label{subsec:motivation}
The motivation for this paper comes from an effort to understand 
Rybnikov's invariant used in \cite{Rybnikov-1,Rybnikov-2,Rybnikov-3} to 
distinguish between the fundamental groups of complements of two 
hyperplane arrangements with the same incidence structure. 
Work of Arnold, Brieskorn, and Orlik--Solomon insures that 
an arrangement complement, $M(\A)$, is rationally formal, 
and that the cohomology ring $H^*(M(\A))$ is determined  
solely by the intersection lattice, $L(\A)$.  Thus, the complements 
of the Rybnikov pair of arrangements share the same rational 
homotopy type; in particular, the respective fundamental groups 
share the same rational associated graded Lie algebras and second 
nilpotent quotients.  Nevertheless, 
the third nilpotent quotients of those two groups are not isomorphic, 
for reasons that are to this date somewhat mysterious, despite 
repeated attempts to elucidate this phenomenon, 
see e.g.~\cite{ACCM-2005,ACCM-2007,Matei,Matei-Suciu}. 

We take here a different approach, closely modeled on 
Rybnikov's original approach from \cite{Rybnikov-1,Rybnikov-2}, 
yet from a more general point of view.  In the process, we develop 
a machinery for determining when a given isomorphism between 
the $n$-th nilpotent quotients of two groups satisfying certain mild 
finiteness and homological assumptions extends to an isomorphism 
between the $(n+1)$-st stages of the respective nilpotent towers.

\subsection{The holonomy map}
\label{subsec:holo}

Let $X$ be a connected CW-complex.  We will assume throughout 
that the first homology group $H_1(X)$ is finitely generated and torsion-free. 
Let $G=\pi_1(X)$ be the fundamental group of $X$, and let 
$\Gamma_n(G)$ denote its lower central series subgroups. 
Finally, let $X\to K(G_{\ab},1)$ be a classifying map corresponding 
to the abelianization homomorphism $G\to G_{\ab}$.  
The induced homomorphism of second homology groups, 
$h\colon H_2(X)\to H_2(G_{\ab})$, is called the \emph{holonomy map}\/ 
of $X$. 

Particularly interesting is the situation when the holonomy map is injective; 
this happens, for instance, when $X$ is the complement
of a complex hyperplane arrangement, or of a `rigid' link, 
or of an arrangement of transverse planes in $\R^4$, \cite{Matei-Suciu}.
Under this injectivity assumption, we show in 
Theorem \ref{theorem-split-exact} that there is a split exact sequence
\begin{equation}
\label{eq:split-exact}
\begin{tikzcd}[column sep=20pt]
0 \ar[r] & \Gamma_n(G)/\Gamma_{n+1}(G)\ar[r] 
& H_2(\GGG{G}{n}) \ar[r] & H_2(X) \ar[r] & 0\, .
\end{tikzcd}
\end{equation}

Many properties of a finitely generated group $G$ are reflected in 
the Lie algebras associated to it. One of those is the 
{\em associated graded Lie algebra}, $\gr(G)$, whose graded pieces 
are defined as $\gr_n(G)=\Gamma_n(G)/\Gamma_{n+1}(G)$, 
and whose Lie bracket is induced from the group commutator.  
The study of the associated graded Lie algebra was initiated in work of 
Magnus \cite{Magnus-1937}, Witt \cite{Witt}, 
Hall  \cite{Hall}, and  Lazard \cite{Lazard}. Much of the power of this approach 
comes from the various connections between the lower central series, 
nilpotent quotients, and group homology, as evidenced in the 
work of  Stallings \cite{Stallings},  Quillen \cite{Quillen-1968}, 
Dwyer \cite{Dwyer-1975}, and many others.  

Another Lie algebra associated to a group $G$ 
is the {\em holonomy Lie algebra},  $\h(G)$, which 
was introduced in work of Chen \cite{Chen73}, Kohno \cite{Kohno-83}, 
and Markl--Papadima \cite{Markl-Papadima}, and studied more recently 
by Papadima--Suciu \cite{Papadima-Suciu-2004} and 
Suciu--Wang \cite{Suciu-Wang-2,Suciu-Wang-1}.  
This Lie algebra depends only on data extracted from the 
cohomology of $G$ in low degrees.  In more detail,  
assuming $G_{\ab}$ is torsion-free, 
$\h(G)$ is defined as the quotient of the free Lie algebra on $G_{\ab}$ 
modulo the ideal generated by the image of the holonomy map, 
$H_2(G)\to H_2(G_{\ab})$.

The holonomy Lie algebra $\h(G)$ may be viewed as a quadratic 
approximation of the  associated graded Lie algebra $\gr(G)$.  
More precisely, there is a canonical epimorphism of graded Lie algebras, 
$\h(G)\twoheadrightarrow \gr(G)$, which is an isomorphism 
in degrees $n\le 2$, but is not necessarily injective in higher degrees 
(see, for instance, the examples in \cite{Suciu-Wang-1} of groups  
that are not graded-formal). Nevertheless, we
show in Theorem \ref{thm:gr3=h3}  that the map $\h_3(G)\to  \gr_3(G)$ 
is an isomorphism under the aforementioned injectivity assumption 
for the holonomy map.

\subsection{The main result}
\label{subsec:main result}
Let $X_a$ and $X_b$ be two path-connected spaces as above.
From \cite{Stallings} it follows that if a map
$f\colon X_a \to X_b$ induces an isomorphism
of first homology groups and an epimorphism of second homology groups,
then $f$ induces an isomorphism $\GGG{G_a}{n}\isom \GGG{G_b}{n}$
for $n \ge 2$, where $G_a$ and $G_b$ denote the fundamental groups
of $X_a$ and $X_b$; respectively.

The main result in this paper gives a necessary and sufficient condition
for a given map of coalgebras, 
\begin{equation}
\label{eq:coalg-map}
\begin{tikzcd}[column sep=20pt]
H_{\le 2}(X_a) \ar[r, "\ov{g}"] & H_{\le 2}(X_b)\, , 
\end{tikzcd}
\end{equation}
with $\ov{g}_1$ an isomorphism and $\ov{g}_2$ an epimorphism,
to \emph{extend}\/ to an isomorphism of nilpotent quotients for a given 
value of $n$; more precisely, that there be an isomorphism
$f_n\colon \GGG{G_a}{n}\to \GGG{G_b}{n}$ such that the diagram
\begin{equation}
\label{eq:h2coalg}
\begin{tikzcd}
H_2(\GGG{G_a}{n}) \ar[r, "(f_n)_\ast"] &  H_2(\GGG{G_b}{n})\\
 H_2(X_a) \ar[u] \ar[r, "\ov{g}_2"] & H_2(X_b)\ar[u]
\end{tikzcd}
\end{equation}
commutes.  To state the result, let  $N$ be a nilpotent group with 
$N \cong \GGG{N}{n}$. 
Assume that the holonomy maps of $X_a$ and $X_b$ are 
injective, and there is a map $\ell_b \colon X_b \to K(N,1)$  
inducing an isomorphism $\GGG{G_b}{n}\isom N$.
We then show in Theorems \ref{thm:ext-ryb} and \ref{thm:main} 
that there is an isomorphism
\begin{equation}
\label{eq:fsub}
\begin{tikzcd}[column sep=20pt]
f_{n+1}\colon \GGG{G_a}{n+1} \ar[r, "\cong"] & \GGG{G_b}{n+1}
\end{tikzcd}
\end{equation}
extending $\ov{g}_2$
if and only if there is a map $\ell_a \colon X_a \to K(N,1)$ inducing 
an isomorphism $\GGG{G_a}{n} \isom N$, and 
a splitting $\sigma$ of the exact sequence \eqref{eq:split-exact} 
such that the following diagram  commutes:
\begin{equation}
\label{diagram-triangle}
\begin{tikzcd}[column sep=1.5pc, row sep=1.6pc]
   & \gr_n(N)\\
   {}\\
   & H_2(N) \ar[uu, "\sigma" ']\\
 H_2(X_a)\ar[rr, "\ov{g}_2", "\cong" ']
 				\ar[uuur, "\mu_a"]
 				\ar[ur, "(\ell_a)_\ast" ']
 		& & H_2(X_b)\, . \ar[uuul, "\mu_b" ']    
		\ar[ul, "(\ell_b)_\ast"]
\end{tikzcd}
\end{equation}

Analogous results in characteristic $p$ are proved in 
Theorems \ref{thm:p-ryb} and \ref{thm:main-p}. 
In the case $n=3$ the obstruction to extending the map 
$\ov{g}_2$ to an isomorphism $\GGG{G_a}{4}\isom \GGG{G_b}{4}$
is computed by generalized Massey triple products.
This result along with further results and applications will be given 
in a subsequent paper.  

\subsection{Hyperplane arrangements}
\label{subsec:hyp}

Returning now to the setting of hyperplane arrangements, let 
$\A$ be a finite set of hyperplanes in some finite-dimensional 
complex vector space.  The complement $M(\A)$, then, has the 
homotopy type of a connected, finite CW-complex. Moreover, 
the cohomology ring $H^*(M(\A))$ is torsion-free and generated 
in degree $1$, and so the holonomy map of $M(\A)$ is injective. 
Consequently, if $G=G(\A)$ is the fundamental group of the complement, 
then $\gr_3(G)\cong \h_3(G)$. 

The second nilpotent quotient of an arrangement group is 
combinatorially determined; that is, if $\A$ and $\B$ are 
two arrangements such that $L_{\le 2} (\A)\cong L_{\le 2} (\B)$, 
then $G(\A)/\Gamma_3(G(\A))\cong G(\B)/\Gamma_3(G(\B))$.
On the other hand, as previously mentioned, Rybnikov showed 
that the next nilpotent quotient, $G(\A)/\Gamma_4(G(\A))$ is not always 
determined by $L_{\le 2}(\A)$.  

The invariant that Rybnikov defined 
in \cite{Rybnikov-1,Rybnikov-2,Rybnikov-3} to prove this result 
comes from the case $n=3$ of the main result in this paper, as follows.
In \cite{Rybnikov-1,Rybnikov-2} it is further assumed that $\h_3(G)$
is torsion-free. Replacing then the modules and maps in Theorem \ref{thm:ext-ryb} 
with their $\Hom$ duals gives the result corresponding to Theorem 2.2 in 
\cite{Rybnikov-1}.  These replacements in Theorem \ref{thm:main} yield 
item 2 of Theorem 12 in \cite{Rybnikov-2}.

Particularly interesting is the class of ``decomposable" hyperplane arrangements.  
Building on work of Papadima and Suciu \cite{Papadima-Suciu-2006} and 
applying Theorem \ref{thm:main}, we prove in Theorem \ref{thm:decomposable-n} 
that, for such an arrangement $\A$, the tower of nilpotent quotients of $G(\A)$ is 
fully determined by the truncated intersection lattice $L_{\le 2}(\A)$.  Our result 
leaves open the question whether the group $G(\A)$ itself is combinatorially 
determined when $\A$ is decomposable. 

\subsection{Organization of the paper}
\label{subsec:org}

The paper is divided into three parts, of roughly equal length. 

The first part deals with the nilpotent quotients and Lie algebras 
associated to a finitely generated group $G$. 
In Section \ref{sect:lcs} we describe the tower of nilpotent 
quotients $\{G/\Gamma_n(G)\}_{n\ge 1}$, while in Section \ref{sect:gr-lie} 
we review the associated graded Lie algebra $\gr(G)$ and the Malcev Lie algebra 
$\mf{m}(G)$.  Finally, in Section \ref{sect:hlie} we discuss the holonomy Lie 
algebra $\h(G)$ and relate it to $\gr(G)$. 

In the second part we reprove and extend Rybnikov's theorem.  
We start in Section \ref{sect:h2nilp} with some preparatory 
material on group extensions, splittings, and $k$-invariants.
The main results, including an extension in characteristic $p$, 
are stated and proved in Section \ref{sect:ryb-gen}.

In the third part we apply our machinery 
to the theory of hyperplane arrangements. We start in 
Section \ref{sect:hyp-arr} with a review of the relevant 
material on the topology and combinatorics of arrangements, 
and give a quick application to Lie algebras associated to arrangement 
groups.  Finally, in Section \ref{sect:nilp-decomp} we show that the nilpotent 
quotients of decomposable arrangement groups are combinatorially 
determined. 

\section{Lower central series and Postnikov towers}
\label{sect:lcs}

In this section we discuss the lower central series and the tower of 
nilpotent quotients of a group. General references include 
the works of P.~Hall \cite{Hall},  Magnus \cite{MKS},
Stallings \cite{Stallings}, and Dwyer \cite{Dwyer-1975}.

\subsection{Lower central series}
\label{subsec:lcs}

Let $G$ be a group.  The {\em lower central series}\/ (LCS) is the 
sequence of subgroups $\{\Gamma_n(G)\}_{n\ge 1}$ 
defined inductively by $\Gamma_1(G)  = G$ and 
\begin{equation}
\label{eq:lcs-series}
\Gamma_{n+1}(G) =[ G,\Gamma_n(G)]
\end{equation}
for $n \ge 1$.  Here, if $H$ and $K$ are subgroups of $G$, then $[H,K]$ denotes the
subgroup of $G$ generated by all elements of the form $[a,b]=aba^{-1}b^{-1}$ 
for $a \in H$ and $b \in K$.  If both $H$ and $K$ are normal subgroups, 
then their commutator $[H,K]$ is again a normal subgroup. 

In our situation, the subgroups $\Gamma_n(G)$ are, in fact, characteristic 
subgroups of $G$. Moreover, the LCS filtration is 
{\em multiplicative}\/, in the sense that, for all $m, n$, 
\begin{equation}
\label{eq:lcs-mult}
[\Gamma_n(G),\Gamma_{m}(G)]\subseteq 
\Gamma_{m+n}(G).
\end{equation}

Note that $\Gamma_2(G)=[G,G]$ is the 
derived subgroup of $G$, and so $G/\Gamma_2(G)=G_{\ab}$, the 
abelianization of $G$.  Furthermore, each term $\Gamma_{n+1}(G)$ 
contains $[ \Gamma_n(G),\Gamma_n(G)]$, and thus the quotient 
group 
\begin{equation}
\label{eq:grng}
\gr_n(G):=\Gamma_n(G)/\Gamma_{n+1}(G)
\end{equation} 
is abelian.

Now let $G=F/R$ be a presentation for our group, with $F$ a free group 
and $R$ a normal subgroup. Then 
$\Gamma_n(G)=\Gamma_n(F)/\Gamma_n(F)\cap R$. 
Moreover, if $G$ is finitely generated, then so are the LCS quotients 
from \eqref{eq:grng}.
We will write $\phi_n(G)=\rank \gr_n(G)$ 
for the ranks of these groups.

For instance, if $F_k$ is the free group on $k$ generators, 
then all its LCS quotients are torsion-free, with ranks  
$\phi_n=\phi_n(F_k)$ given by $\prod_{n=1}^{\infty}(1-t^n)^{\phi_n}=
1-kt$, or, equivalently, $\phi_n=\tfrac{1}{k}\sum_{d\mid k} \mu(d) n^{k/d}$, 
where $\mu$ denotes the M\"{o}bius function.

\subsection{Nilpotent quotients}
\label{subsec:nilp-quot}
It is readily seen that $G/\Gamma_{n+1}(G)$ is a nilpotent group, 
and in fact, the maximal $n$-step nilpotent quotient of $G$.  
Letting $q_{n} \colon \Gg{n+1}\to \Gg{n}$ be the projection map, 
we obtain a tower of nilpotent groups,  
\begin{equation}
\label{eq:nilpotent tower}
\begin{tikzcd}
\cdots \ar[r] 
    & \Gg{4}  \ar[r, "q_3"]
    & \Gg{3}  \ar[r, "q_2"]
    & \Gg{2}
\end{tikzcd}.
\end{equation}

For each $n\ge 1$, we have  a central extension, 
\begin{equation}
\label{eq:extension}
\begin{tikzcd}[column sep=20pt]
0\ar[r]
	&\gr_{n}(G) \ar[r]
	& G/\Gamma_{n+1}(G) \ar[r, "q_n"]
	& G/\Gamma_n(G)\ar[r]
	& 0
\end{tikzcd}.
\end{equation}
Passing to classifying spaces, we obtain a 
commutative diagram,
\begin{equation}
\label{equation-extension}
\begin{tikzcd}[row sep=2pc, column sep=2.2pc]
& K(\Gg{n+1},1) \ar[d, "\pi_{n}"]\\
K(G,1) \ar[ur, "\psi_{n+1}"]
\ar[r, "\psi_n"]
& K(\Gg{n},1),
\end{tikzcd}
\end{equation}
where $\psi_n$ corresponds to the projection 
$G\to \Gg{n}$ and $\pi_n$ corresponds to the projection $q_n$.  
Note that $\pi_{n}$ may be viewed as the fibration with fiber 
$K(\gr_n(G),1)$ obtained as the pullback of the pathspace 
fibration with base $K(\gr_n(G),2)$ via a $k$-invariant 
\begin{equation}
\label{eq:k-inv}
\begin{tikzcd}[column sep=20pt]
\chi_n \colon K(\Gg{n},1) \ar[r]& K(\gr_n(G),2)
\end{tikzcd}.
\end{equation}

\subsection{Postnikov tower and the holonomy map}
\label{subsec:holo map}

Now let $X$ be a connected CW-complex, and let $G = \pi_1(X)$ 
be its fundamental group.
An Eilenberg--MacLane space $K(G,1)$ can be constructed
by adding to $X$ cells of dimension three or more; let 
$\iota\colon X\to K(G,1)$ be the inclusion map. 
For each $n\ge 1$,  let $h_n\colon X\to K(\Gg{n},1)$ 
be the composite $\psi_n \circ \iota$. 
This gives the following \emph{Postnikov tower}\/ of fibrations:
\vspace*{-8pt}
\begin{equation}
\label{eq:postnikov}
\begin{tikzcd}[row sep=24pt, column sep=60pt]
&\ar[d, dotted] \\
&  K(G/\G_4(G),1) \ar[d, "\pi_3"]\\
&  K(G/\G_3(G),1) \ar[d, "\pi_2"]\\
X \ar[r, "h_2" '] \ar[ur, "h_3" '] \ar[uur, "h_4"]
&  K(G/\G_2(G),1)
\end{tikzcd}
\end{equation}

We take now homology with coefficients in $\Z$. 
From the above discussion, we deduce the well-known fact that 
the map $\iota\colon X\to K(G,1)$ induces an isomorphism 
$\iota_*\colon H_1(X)\isom H_1(G)$ and an epimorphism 
$\iota_*\colon H_2(X)\surj H_2(G)$.  

Consider now the 
Lyndon--Hochschild--Serre spectral sequence
defined in \cite{Cartan},
\begin{equation}
\label{eq:lhs}
H_p(G/N;H_q(N;M)) \Rightarrow H_n(G;M)\, ,
\end{equation}
where $G$ is a group, $N$ is a normal subgroup of $G$, and 
$M$ is a $G$-module.
For the central central extension \eqref{eq:extension}, the 
$5$-term exact sequence arising from the terms of low degree 
(see e.g.~\cite[Theorem~2.1]{Stallings}) reduces to a short 
exact sequence,
\begin{equation}
\label{equation-Stallings}
\begin{tikzcd}
H_2(G) \ar[r, "(q_n)_\ast"]
	&H_2(\Gg{n}) \ar[r, "\chi_n"]
	& \gr_n(G) \ar[r]
	& 0 
\end{tikzcd},
\end{equation}
where the map $\chi_n$ corresponds to the $k$-invariant from \eqref{eq:k-inv} 
via the Universal Coefficient Theorem. Using now the surjectivity of the map 
$\iota_*\colon H_2(X)\to H_2(G)$ we obtain an exact sequence,
\begin{equation}
\label{eq:h2-gr}
\begin{tikzcd}
H_2(X) \ar[r, "(h_n)_\ast"]
	&H_2(\Gg{n}) \ar[r, "\chi_n"]
	& \gr_n(G) \ar[r]
	& 0 
\end{tikzcd},
\end{equation}
In general, the sequence in \eqref{eq:h2-gr} is natural but not split exact. 
We call the homomorphism
\begin{equation}
\label{eq:holonomy}
\begin{tikzcd}[column sep=20pt]
(h_2)_{\ast}\colon H_2(X) \ar[r]
	& H_2(\Gg{2}) \cong H_1(X) \wedge H_1(X)
\end{tikzcd}
\end{equation}
the \emph{holonomy map}\/ of $X$.  The following lemma easily follows 
from the definitions and the Universal Coefficient Theorem.

\begin{lemma}
\label{lem:cup-holo}
Suppose $H_1(X)$ is finitely generated and 
torsion-free.  Then the holonomy map of $X$ is dual to the cup-product map
\begin{equation}
\label{eq:cup}
\begin{tikzcd}[column sep=18pt]
\cup\colon H^1(X) \wedge H^1(X) \ar[r] & H^2(X)\, .
\end{tikzcd}
\end{equation}
Consequently, if the cup-product map from \eqref{eq:cup} is surjective,
the holonomy map from \eqref{eq:holonomy} is injective.
\end{lemma}

As the next example shows, the converse of the last statement does not hold.

\begin{example}
\label{ex:rat-torus}
Let $X$ be the connected, $2$-dimensional CW-complex associated to 
the group $G$ with presentation $G=\langle x,y\mid x^2yx^{-2}y^{-1}=1\rangle$. 
Clearly, $H_1(X)=\Z^2$ and $H_2(X)=\Z$.  With these identifications, the holonomy 
map $(h_2)_*\colon \Z\to \Z$ is multiplication by $2$, and thus injective, 
while the cup-product map $\cup\colon \Z\to \Z$ is also multiplication 
by $2$, and thus not surjective.
\end{example}

\section{Associated graded Lie algebras and Malcev Lie algebras}
\label{sect:gr-lie}

\subsection{The associated graded Lie algebra of a group}
\label{subsec:gr}

Given a group $G$, we let $\gr(G)$ 
be the direct sum of the successive quotients of the lower 
central series of $G$; that is, 
\begin{equation}
\label{eq:lcs}
\gr(G) = \bigoplus_{n\ge 1} \Gamma_n(G)/\Gamma_{n+1}(G)\, .
\end{equation}
 The map $a\otimes b \mapsto aba^{-1}b^{-1}$
induces homomorphisms $[\, ,\, ]\colon \gr_m(G) \otimes \gr_n(G) \to \gr_{m+n}(G)$.
It is readily seen that the following ``Witt--Hall identities" hold in $G$:
\begin{equation}
\label{eq:witt-hall}
[ab, c]= {}^a [b,c]\cdot [a,c],\qquad 
[ {}^b a, [c, b]] \cdot [{}^c b, [a, c]] \cdot [{}^a c,[b, a]]= 1,
\end{equation}
where ${}^a b= aba^{-1}$.  
It follows that $\gr(G)$, endowed with the aforementioned bracket, 
has the structure of a graded Lie algebra, see for instance \cite{MKS,Serre}. 
The construction is functorial:
every group homomorphism $f\colon G\to H$ induces a morphism 
of graded Lie algebras, $\gr(f)\colon \gr(G)\to \gr(H)$.

If $F$ is a free group, then, as shown by Magnus and Witt, $\gr(F)$ 
is the free Lie algebra on the same set of generators as $F$; in particular, 
if $F=F_\ell$, then $\gr(F)=\Lie(\Z^\ell)$, the free Lie algebra of rank 
$\ell$. 

\subsection{Injective holonomy map and an exact sequence}
\label{subsec:inj-holo}

Once again, let $X$ be a path-connected space, with fundamental group 
$G=\pi_1(X)$.

\begin{theorem}
\label{theorem-split-exact}
Assume that the group $H_1(X)$ is finitely generated, torsion-free,
and the holonomy map $(h_2)_{\ast} \colon H_2(X) \to 
H_1(X) \wedge H_1(X)$ from \eqref{eq:holonomy} is injective.
For each $n \ge 2$, there is then a natural, split exact sequence
\begin{equation}
\label{equation-from-Serre}
\begin{tikzcd}[column sep=20pt]
0\ar[r]
	&\gr_{n+1}(G) \ar[r, "i"]
	& H_2(\Gg{n+1}) \ar[r, "\pi"]
	& H_2(X)\ar[r]
	& 0 \, .
\end{tikzcd}
\end{equation}
\end{theorem}

\begin{proof}
Recall that given a fibration of CW-complexes with base $B$, fiber $F$,
and total space $E$, the filtration of $C_\ast(E)$ by the inverse images
of the skeleta of $B$ gives a homology Serre spectral 
sequence with differentials
\begin{equation}
\label{eq:serre-homology}
\begin{tikzcd}
E^r_{s,t}\ar[r, "d^r"] & E^r_{s-r,t+r-1}\, .
\end{tikzcd}
\end{equation}
Furthermore, if the fundamental group of the base acts trivially on the fibers, then
\begin{equation}
\label{eq:e2page}
E^2_{s,t} = H_s(B) \otimes H_t(F)\, .
\end{equation}
For the fibration \eqref{equation-extension}, the kernel of the
differential $d^2\colon H_2(B) \to H_1(F)$ is by equation 
\eqref{equation-Stallings} the image of $H_2(X)$ in
$H_2(\Gg{n})$.  In turn, this image can be identified 
with $H_2(X)$, since by assumption the holonomy 
map is a monomorphism. This gives an exact
sequence
\begin{equation}
\label{equation-split}
\begin{tikzcd}[column sep=20pt]
0 \ar[r] & F^1 \ar[r] &H_2(\Gg{n+1}) \ar[r] & H_2(X) \ar[r] & 0\, ,
\end{tikzcd}
\end{equation}
where $F^1$ denotes the image in $H_2(\Gg{n+1})$ 
of the inverse image of the 1-skeleton in $K(\Gg{n},1)$.
It follows that a map $(h_{n+1})_\ast$ gives a right splitting of
the exact sequence \eqref{equation-split}. The result now
follows from \eqref{eq:h2-gr}.
\end{proof}

In particular, under the above hypothesis there is a natural exact sequence
\begin{equation}
\label{equation-from-Serre-2}
\begin{tikzcd}[column sep=20pt]
0 \ar[r] 
	&\gr_3(G) \ar[r]
	& H_2(\Gg{3}) \ar[r]
	&  H_2(X) \ar[r]
	& 0\, .
\end{tikzcd}
\end{equation}
Furthermore, this sequence is split exact. We do not claim that there
is a natural splitting of the exact sequence \eqref{equation-from-Serre-2}, 
or of the other exact sequences from \eqref{equation-from-Serre}.

\subsection{Malcev completion and the Malcev Lie algebra}
\label{subsec:malcev}

In \cite{Malcev}, Malcev established a one-to-one correspondence between certain
nilpotent Lie algebras over $\Q$, and nilpotent groups over $\Q$, 
leading to the Malcev Lie algebra of a group. 
This was extended by Lazard \cite{Lazard} to groups with enough divisibility
in central series subgroups to establish a one-to-one correspondence 
between a wider class of nilpotent groups and Lie algebras. 
An important next step was taken by Quillen, who established 
in \cite{Quillen-1969} an equivalence between rational homotopy 
theory and the homotopy theory of reduced differential graded Lie algebras over 
$\Q$ with Malcev Lie algebras as the equivalent of the rational fundamental group.
This was extended by Dwyer in \cite{Dwyer-1979} to an equivalence between the
tame homotopy theory of $2$-connected spaces and differential graded
Lazard Lie algebras. 

In more detail, assume  that $G$ is a finitely generated group.  
It is then possible to replace each nilpotent quotient $N_n=\Gg{n}$ by 
$N_n \otimes \Q$, the (rationally defined) nilpotent Lie group associated to 
the discrete, torsion-free nilpotent group $N_n/{\rm tors}(N_n)$.
The corresponding inverse limit, 
\begin{equation}
\label{eq:mcomp}
\mathfrak{M}(G)=\varprojlim_n\, (\Gg{n}\otimes \Q),
\end{equation}
is a prounipotent, filtered Lie group over $\Q$, which is 
called the {\em prounipotent completion}, or {\em Malcev completion}\/ of $G$.  

Let us denote by $\mc{L}(K)$ the Lie algebra of a Lie group $K$.
The pronilpotent Lie algebra 
\begin{equation}
\label{eq:Malcev Lie}
\mf{m}(G)=\varprojlim_n \mc{L}(\Gg{n}\otimes \Q),
\end{equation}
endowed with the inverse limit filtration, is called the 
{\em Malcev Lie algebra}\/ of $G$.  
By construction, $\mf{m}(-)$ is a functor 
from the category of finitely generated groups to the category of
complete, separated, filtered Lie algebras over $\Q$.

In \cite{Quillen-1968,Quillen-1969}, 
Quillen showed that $\mf{m}(G)$ is the set of all primitive elements in
$\widehat{\Q{G}}$, the completion of the group algebra of $G$ 
with respect to the filtration by powers of the augmentation ideal, 
and that the associated graded Lie algebra of $\mf{m}(G)$ with respect to the 
inverse limit filtration is isomorphic to $\gr(G;\Q)$. Furthermore, 
the set of all group-like elements in $ \widehat{\Q{G}}$, 
with multiplication and filtration inherited from $\widehat{\Q{G}}$, 
forms a complete, filtered group isomorphic to $\mathfrak{M}(G)$.

\subsection{The Sullivan minimal model}
\label{subsec:minmodel}

In a seminal paper \cite{Sullivan}, Sullivan showed that commutative 
differential graded algebras (cdgas) of differential forms over $\Q$ 
can be used to model rational homotopy theory.  From this 
perspective the commutative differential graded algebra corresponding
to the Malcev Lie algebra of a group is obtained by taking the
free commutative differential graded algebra $\Hom$ dual to nilpotent
quotients of the Lie algebra and passing to the limit. 

More precisely, Sullivan associated to each space $X$ a 
cdga over the rationals, denoted $A_{\PL}^{*}(X)$, 
for which there is an isomorphism $H^{*}(A_{\PL}(X)) \cong H^{*}(X,\Q)$ 
under which induced homomorphisms in cohomology 
correspond. A space $X$ is said to be {\em formal}\/ if 
$A_{\PL}^{*}(X)\cong (H^{*}(X;\Q),d=0)$, i.e., Sullivan's 
algebra can be connected by a zig-zag of quasi-isomorphisms 
to the rational cohomology ring of $X$, endowed with the 
zero differential. 

A \emph{Hirsch extension}\/ (of degree $i$) is a cdga inclusion 
$(A,d)\inj (A\otimes \bigwedge(V),d)$, where 
$V$ is a $\Q$-vector space concentrated in 
degree $i$, while $\bigwedge(V)$ is the free graded-commutative 
algebra generated by $V$, and $d$ sends $V$ into $A^{i+1}$. 
A cdga $(A,d)$ is called \emph{minimal}\/ if $A$ is connected (i.e., $A^0=\Q$), 
and  the following two conditions are satisfied:
(1) $A=\bigcup_{j\geq0}A_j$, where $A_0=\Q$ and each 
$A_{j}$ is a Hirsch extension of $A_{j-1}$; 
(2) the differential is decomposable, i.e.,  
$dA\subset A^+\wedge A^+$, where $A^+=\bigoplus_{i\geq1}A^i$.
A basic result of Sullivan  \cite{Sullivan} and Morgan \cite{Morgan} asserts 
the following: Each connected cdga $(A,d)$ has a minimal model $\mathcal{M}(A)$, 
unique up to isomorphism. 

Suppose now that $X$ is a connected CW-complex with finitely many $1$-cells. 
Then the Lie algebra dual to the first stage of the minimal model associated to 
$A_{\PL}^{*}(X)$ is isomorphic to the Malcev Lie algebra $\mf{m}(\pi_1(X))$.
A finitely generated group $G$ is said to be \emph{$1$-formal}\/ 
(over $\Q$) if it has a classifying space $K(G,1)$ which is $1$-formal, 
or, equivalently, if the Malcev Lie algebra $\mf{m}(G)$ is the completion 
of a quadratic Lie algebra.  
For a comprehensive discussion of all these notions and more we refer to the 
monographs \cite{FHT,FHT2} and to the papers 
\cite{Papadima-Suciu-2009,Suciu-Wang-1}.

The next step is to look for invariants beyond rational homotopy theory. 
Chen, Fox, and Lyndon \cite{Chen-Fox-Lyndon} gave examples using 
Fox derivatives to find groups whose successive quotients in
the lower central series have torsion. Stallings \cite{Stallings} related
homological properties of a group to successive  quotients in the lower central
series and also to successive quotients in a mod $p$ descending series. 
Building on this work, Dwyer \cite{Dwyer-1975}  related Massey products
in the cohomology of a group to properties of the quotients in the lower 
central series and also to a mod $p$ central series different than the one 
used by Stallings.  In \cite{Cenkl-1980-2,Cenkl-1983,Cenkl}, Cenkl and Porter used 
a commutative algebra of differential forms to model tame homotopy 
theory and the Lazard Lie algebra completion of the fundamental group. 

Massey products were defined in \cite{Massey-1958} and applied to prove the 
Jacobi identity for Whitehead products. Porter \cite{Porter} gave a general formula for
Massey products in a commutator relators group in terms of coefficients in the 
Magnus expansions of the relators and provided applications to links.
In \cite{Matei}, Matei gave examples of complements of hyperplane arrangements 
with nonzero mod $p$ Massey products. This shows that while arrangement 
complements are formal over the 
rationals---and hence all Massey products with rational coefficients contain 
zero---they are not necessarily formal over the integers. Recently, Salvatore 
\cite{Salvatore} gave examples of configuration spaces that are not formal
over the integers.

\section{Holonomy Lie algebras}
\label{sect:hlie}

Among all the Lie algebras one can associate to a group, the 
simplest is the holonomy Lie algebra, which only depends 
on data extracted from cohomology in low degrees.  In this section 
we shed new light on the relationship between the holonomy Lie 
algebra and the associated graded Lie algebra of a group. 

\subsection{The holonomy Lie algebra of a group}
\label{subsec:holo-lie}

Let $G$ be a group, and fix a coefficient ring $\k$, which 
we will take to be either a field or the integers.  We will assume 
throughout that $H=H_1(G,\k)$ is a finitely generated $\k$-module; 
moreover, when $\k=\Z$, we will assume for simplicity that 
$H$ is torsion-free.  We let $\Lie(H)$ denote the free Lie 
algebra on the free $\k$-module 
$H$; note that $\Lie_1(H)=H$ and $\Lie_2(H)=H\wedge H$. 

Following \cite{Chen73,Kohno-83,Markl-Papadima,Papadima-Suciu-2004,Suciu-Wang-2,Suciu-Wang-1}, 
we define $\h(G,\k)$, the {\em holonomy Lie algebra}\/ of $G$, 
as the quotient of the free Lie algebra on $H_1(G,\k)$ by the Lie 
ideal generated by the image of the holonomy map, 
$(h_2)_{\ast}\colon H_2(G,\k) \to H_1(G,\k) \wedge H_1(G,\k)$:
\begin{equation}
\label{eq:holo-lie}
\h(G,\k) = \Lie (H_1(G,\k))/ \text{ideal}(\im ((h_2)_{\ast}))\, .
\end{equation}

The holonomy Lie algebra of $G$ is a quadratic Lie algebra: it is generated 
in degree $1$ by by $\h_1(G,\k)=H_1(G,\k)$, and all the relations 
are in degree $2$.  For $\k=\Z$, we simply write  $\h(G)=\h(G,\Z)$.  
Clearly, the construction is functorial: every group homomorphism 
$f\colon G\to H$ induces a morphism of graded Lie algebras, 
$\h(f)\colon \h(G,\k)\to \h(H,\k)$.  

As noted in \cite{Suciu-Wang-2}, the projection map 
$\psi_n\colon G\surj G/\Gamma_n(G)$ induces an isomorphism 
$\h(\psi_n)\colon \h(G)\isom \h(G/\Gamma_n(G))$ for all $n\ge 3$. 
In particular, the holonomy Lie algebra of  $G$ depends 
only on its second nilpotent quotient, $G/\Gamma_3 (G)$. 

In a completely analogous fashion, one may define the holonomy 
Lie algebra $\h(A)$ of a graded, graded-commutative $\k$-algebra $A$, 
provided that $A^0=\k$ and $A^1$ is finite-dimensional (and torsion-free 
if $\k=\Z$). It is readily seen that $\h(A)=\h(A^{\le 2})$. Moreover, 
if $G$ is a group as above, $\h(G)=\h(H^*(G;\k))$. In fact, if $X$ 
is any path-connected space with $G=\pi_1(X)$, then we may 
define $\h(X):=\h(H^*(X;\k))$, after which it is easily verified that 
$\h(X)\cong \h(G)$.

On a historical note, the holonomy Lie algebra of a group $G$ 
was first defined (over $\k=\Q$) by Chen in \cite{Chen73}, 
and later considered by Kohno in \cite{Kohno-83} 
in the case when $G$ is the fundamental group of the complement 
of a complex projective hypersurface.  In \cite{Markl-Papadima}, 
Markl and Papadima extended the definition of the holonomy 
Lie algebra to integral coefficients.  Further in-depth studies 
were done by Papadima--Suciu  \cite{Papadima-Suciu-2004}
and Suciu--Wang \cite{Suciu-Wang-2,Suciu-Wang-1}; in particular, 
the more general case when the group $H=H_1(G,\Z)$ is allowed 
to have torsion is treated in \cite{Suciu-Wang-2}.

\subsection{A comparison map}
\label{subsec:comparison}
Now set $\gr_n(G,\k) =\gr_n(G)\otimes \k$, and let 
$\gr(G,\k) = \bigoplus_{n\ge 1} \gr_n(G,\k)$ 
be the associated graded Lie algebra of $G$ over $\k$. 
As shown in  
\cite{Markl-Papadima,Papadima-Suciu-2004,Suciu-Wang-2,Suciu-Wang-1}, 
there is a (functorially defined) surjective morphism 
of graded Lie algebras, 
\begin{equation}
\label{eq:holo-gr}
\begin{tikzcd}[column sep=16pt]
\h(G,\k) \ar[r, two heads] & \gr(G;\k)\,,
\end{tikzcd}
\end{equation}
which restricts to isomorphisms  $\h_n(G,\k) \to \gr_n(G;\k)$ for $n\le 2$. 

The above map is an isomorphism if the group $G$ is 
$1$-formal over a field $\k$ of characteristic $0$, 
but in general it fails to be injective in degrees $n\ge 3$. 
Nevertheless, for a large class of (not necessarily 
$1$-formal) groups, the map $\h_3(G,\k) \to \gr_3(G;\k)$ 
is an isomorphism, even for $\k=\Z$.  This will be made more 
precise in Theorem \ref{thm:gr3=h3} below.

\subsection{Another exact sequence}
\label{subsec:exseq}

As before, let $h_2\colon X \to K(\Gg{2},1)$ be the continuous map 
induced by the projection of $G$ to $\Gg{2}$, and let 
$(h_2)_*\colon H_2(X) \to H_2(\Gg{2})$ be the corresponding 
holonomy map. 

\begin{theorem}
\label{thm:sss}
If $H_1(X)$ is finitely generated and torsion-free, then
there is an exact sequence
\begin{equation}
\label{eq:h3}
\begin{tikzcd}[column sep=18pt]
0 \ar[r] & \h_3(G) \ar[r] & H_2(\Gg{3}) \ar[r]
	& H_2(X)/(\ker (h_2)_\ast) \ar[r] & 0\, .
\end{tikzcd}
\end{equation}
\end{theorem}

\begin{proof}
We shall make use of the homology Serre spectral sequence 
associated to the extension 
\begin{equation}
\label{eq:h3h2}
\begin{tikzcd}[column sep=18pt]
0 \ar[r] & \gr_2(G) \ar[r] & \Gg{3}\ar[r]
	& \Gg{2} \ar[r] & 0\, .
\end{tikzcd}
\end{equation}
The $E^2$ page of this spectral sequence depicted in 
diagram \eqref{eq:serre} below.
\begin{equation}
\label{eq:serre}
\begin{gathered}
\begin{sseqpage}[ homological Serre grading ]
\class(0,1)
\class(1,1)
\class(2,0)
\class(0,2)
\class(2,1)
\class(3,0)
\d ["d^{\,2}_{2,0}" { pos = 0.75, yshift = -24pt } ] 2 (2,0)
\d ["d^{\,2}_{3,0}" { pos = 0.5, xshift = 3pt, yshift = -3pt } ] 2 (3,0)
\d  ["d^{\,2}_{2,1}" ]  2 (2,1)
\end{sseqpage}
\end{gathered}
\end{equation}

Assume first that $\h_2$ is torsion-free, where $\h=\h(G)$.
Our hypotheses on the abelian groups $\Gg{2}=\h_1=H_1(X)$ and $\gr_2(G)=\h_2$ imply 
that all the terms $E^2_{p,q}=H_p(\Gg{2},H_q(\h_2))$ are finitely generated and 
torsion-free, and hence, Hom dual to the $E_2$ terms and differentials $d_2$ 
in the cohomology spectral sequence associated to extension \eqref{eq:h3h2}. 
Since the $E_2$ terms and differentials 
form a commutative differential graded algebra,
the differentials $d_2$ are determined by the differential $d_2^{\,0,1}$
dual to $d^{\,2}_{2,0}\colon H_2(\Gg{2}) \to H_1 (\h_2)$, which is given by
\begin{equation}\label{eq:sss-1}
d^{\,2}_{2,0}(x_a \wedge x_b) = - [x_a,x_b],
\end{equation}
where $[\cdot , \cdot ]$ denotes the bracket map from
$\h_1 \wedge \h_1$ to $\h_2$.
Computing the differential $d^{1,1}_2$ and $d^{0,2}_2$ and 
then taking the dual maps gives the following: 
\begin{equation}
\label{eq:sss-2}
{d^2_{3,0}}(x_a \wedge x_b \wedge x_c) 
	 = x_a \wedge [x_b,x_c] - x_b \wedge [x_a,x_c] + x_c \wedge [x_a,x_b]
\end{equation}
for $x_a, x_b, x_c \in \h_1$ and 	
\begin{equation}
\label{eq:sss-3}	
{d^2_{2,1}}(x_a \wedge x_b \wedge x_d)
 	=  - [x_a,x_b] \wedge x_d
\end{equation}
for $x_a, x_b \in \h_1$ and $x_d \in \h_2$.

 From \eqref{eq:sss-1} and \eqref{eq:sss-3} it follows that
 $E^3_{0,2} = 0$, while from \eqref{eq:sss-2} it follows that 
 $E^3_{1,1} = \h_3$.  Now note that $E^3_{2,0}$ is the kernel of the map
 $d^2_{2,0}$ in equation \eqref{eq:sss-1}. From the formula for $d^2_{2,0}$
 in \eqref{eq:sss-1} and the exact sequence \eqref{equation-Stallings}, 
 it follows that  $E^3_{2,0}$ is the image of
 $H_2(X)$ in $H_2(\GG{2})$, which is $ H_2(X)/(\ker (h_2)_\ast)$. 
 
Looking at the domains and ranges of the higher-order differentials in the spectral
 sequence, we see that since $E^3_{0,2} = 0$, it follows that
 $E^3_{p,q} = E^\infty_{p,q}$ for $p+q \le 2$.   We conclude that 
\begin{equation}
\label{eq:sss-e-infty}
E^\infty_{0,2} = 0,\qquad E^\infty_{1,1} = \h_3,\quad\text{and}\quad
E^\infty_{2,0} = H_2(X)/(\ker (h_2)_\ast)\, .
\end{equation}
Equation \eqref{eq:h3} now follows, and the proof of the lemma is complete
in the case where $\h_2(G)$ is torsion-free.

In the case where $\h_2$ has torsion, let $x_1, \ldots , x_\ell$ 
be elements in $G$ that project to a basis for $\Gg{2}$.
Set $\mcf$ equal to the free group on the generators $x_i$, 
and note that $\h_1(\mcf)=\h_1$ and $\h_2(\mcf)=\h_1\wedge \h_1$. 
The identity map of generators gives a map of central extensions,
\begin{equation}
\begin{gathered}
\begin{tikzcd}
0 \ar[r]  & \h_1 \wedge \h_1 \ar[r] \ar[d, "\lbrack\: {,} \:\rbrack" ']
	& \GGG{\mcf}{3}\ar[r] \ar[d]
	& \GGG{\mcf}{2} \ar[r] \ar[d]
	& 0\phantom{\, ,} \\
0 \ar[r] & \h_2\ar[r]
	&	\GG{3} \ar[r]
	& \GG{2} \ar[r]
	& 0\, ,
\end{tikzcd}
\end{gathered}
\end{equation}
and hence a map of the respective homology spectral sequences.
By the argument above, equations \eqref{eq:sss-1}, \eqref{eq:sss-2}, 
and \eqref{eq:sss-3} hold in the spectral sequence for $\mcf$ and hence in the
spectral sequence for $G$. Moreover, each of the maps of the $E^2$ terms
involved in these equations is onto, so it follows that the equations in
\eqref{eq:sss-e-infty} hold for $G$ as well.  This completes the proof.
\end{proof}

 \begin{remark}
\label{rem:h2-lie-gp}
If the group $\h_2=\h_2(G)$ is torsion-free, then the commutative differential graded algebra 
$(E_2,d_2)$ is the Chevalley--Eilenberg cochain complex \cite{Cartan} 
of the Lie algebra $\h/\Gamma_3(\h)$. 
Since $E^3_{p,q} = E^\infty_{p,q}$ for $p+q = 2$, 
it follows that the Lie algebra homology group $H_2(\GGG{\h}{3})$ is
isomorphic to $H_2(\Gg{3})$.
\end{remark}

\subsection{Identifying $\h_3(G)$ with $\gr_3(G)$}
\label{subsec:h3=gr3}

We are now ready to state and prove the main result of this section. 
A proof of this theorem was first sketched by Rybnikov in \cite[\S 3]{Rybnikov-2}; 
we provide here an alternate proof, with full details.

\begin{theorem}
\label{thm:gr3=h3}
Suppose $H=H_1(G;\Z)$ is a finitely-generated, free abelian group, and the
holonomy map $(h_2)_{\ast}\colon H_2(G) \to H \wedge H$ is injective.  Then 
the canonical projection $\h_3(G)\to  \gr_3(G)$ is an isomorphism.
\end{theorem}

\begin{proof}
Consider the homology spectral sequence of the exact sequence 
from \eqref{eq:h3h2}, whose $E^2$ page 
is pictured in diagram \eqref{eq:serre}.  As in 
the proof of Theorem \ref{theorem-split-exact}, 
let $F^1$ denote the image in $H_2(\Gg{3})$ 
of the inverse image of the $1$-skeleton of $K(\Gg{2},1)$.

The proof of Theorem \ref{thm:sss} shows that if $H$ is torsion-free, 
then $F^1 \cong \h_3(G)$.

The proof of Theorem \ref{theorem-split-exact} shows that
if $H$ is torsion-free and the holonomy map $(h_2)_{\ast}$ 
is injective, then $F^1 \cong \gr_3(G)$, and the result follows.
\end{proof}

\section{Second cohomology of nilpotent groups and associated $k$-invariants}
\label{sect:h2nilp}

The purpose of this section is to use the exact sequence in
equation \eqref{equation-from-Serre} to relate 
homomorphisms from $H_2(X)$ to $\gr_n(G)$ to
the possible $k$-invariants of the
central extension of $\Gg{n}$ to $\Gg{n+1}$ 
from \eqref{eq:extension}.
Throughout this section, homology will be taken with coefficients in $\Z$.

In general, given an exact sequence of abelian groups
\begin{equation}
\label{equation-exact}
\begin{tikzcd}[column sep=20pt]
0 \ar[r] & A \ar[r, "i"] & B \ar[r, "j"] & C \ar[r] & 0\, ,
\end{tikzcd}
\end{equation}
a map $\sigma\colon B \to A$ with $\sigma \circ i = \text{id}_A$ is called a 
\emph{left splitting} and a map $h\colon C \to B$ with $j \circ h = \text{id}_C$
is called a \emph{right splitting}. Recall the following well-known fact:
The exact sequence in \eqref{equation-exact} splits (either on the left 
or the right) if and only if $B \cong A \oplus C$. Furthermore, the direct
sum decompositions of this sort are in one-to-one correspondence with
splitting maps $B \to A$ (or $C \to B$). Moreover, as shown in the proof
of Lemma \ref{lemma-bijection} below, a splitting yields a bijection between
all splittings and the set of homomorphisms from $C$ to $A$. 

Once again, let $X$ be a path-connected space such that $H_1(X)$ is finitely generated 
and torsion-free, and such that the holonomy map $(h_2)_{\ast} \colon H_2(X) \to 
H_1(X) \wedge H_1(X)$ is injective.  Set $G=\pi_1(X)$. 
Recall from Theorem \ref{theorem-split-exact} that for $n \ge 3$ there is a 
split exact sequence
\begin{equation}
\label{equation-from-Serre-3}
\begin{tikzcd}
0\ar[r]
	&\gr_{n}(G) \ar[r, "i"]
	& H_2(\Gg{n}) \ar[r, "\pi"]
	& H_2(X)\ar[r]
	& 0
\end{tikzcd}
\end{equation}
and from equation \eqref{equation-Stallings}, the $k$-invariant
$\chi_n$ gives a splitting; that is, in the diagram
\begin{equation}
\label{equation-from-Serre-splitting}
\begin{tikzcd}
0\ar[r]
	&\gr_{n}(G) \ar[r, "i"]
	& H_2(\Gg{n}) \ar[r, "\pi"] \ar[l, bend left=22, "\chi_n"]
	& H_2(X)\ar[r] \ar[l, bend left=22, "(h_n)_\ast"]
	& 0
\end{tikzcd}
\end{equation}
the map $\chi_n \circ i$ is the identity on $\gr_n(G)$, while 
$\pi \circ (h_n)_\ast$ is the identity on $H_2(X)$ and 
$\ker \chi_n = \im (h_n)_\ast$.

\begin{lemma}
\label{lemma-bijection}
For $n \ge 3$, any homomorphism 
$\sigma\colon H_2(\Gg{n}) \to \gr_n(G)$ with $\sigma \circ i$ 
equal to the identity on $\gr_n(G)$ yields a bijection between splittings
of the exact sequence \eqref{equation-from-Serre-3} and elements in
$\Hom(H_2(X), \gr_n(G))$.
\end{lemma}

\begin{proof}
The map $\sigma$ gives an isomorphism between $H_2(\Gg{n})$ and 
$ \gr_n(G) \oplus H_2(X)$. Without loss of generality, we can assume that
via this isomorphism the inclusion $i$ and the projection $\pi$ in 
\eqref{equation-from-Serre-splitting} correspond to the maps
$\widetilde{i}$ and $\widetilde{\pi}$ in the diagram below
\begin{equation}
\label{equation-proof}
\begin{tikzcd}[column sep=22pt]
0\ar[r]
	&\gr_{n}(G) \ar[r, "\widetilde{i}"]
	& \gr_n(G) \oplus H_2(X) \ar[r, "\widetilde{\pi}"]
	& H_2(X)\ar[r]
	& 0\, ,
\end{tikzcd}
\end{equation}
where $\widetilde{i}$ is the inclusion into the first coordinate and 
$\widetilde{\pi}$ is the projection onto the second coordinate.

An element $\lambda \in \Hom(H_2(X), \gr_n(G))$ determines a splitting
of \eqref{equation-proof} as follows.
Given $\lambda$, define a map $h\colon H_2(X) \to \gr_n(G) \oplus H_2(X)$ by
$h(c) = (\lambda(c), c)$, and define 
$\chi\colon \gr_n(G) \oplus H_2(X) \to \gr_n(G)$ by 
$\chi(x, c) = x - \lambda(c)$. 

It is straightforward to check that in the diagram
\begin{equation}
\label{equation-with-splitting-maps}
\begin{tikzcd}[column sep=20pt]
0\ar[r]
	&\gr_{n}(G) \ar[r, "\widetilde{i}"]
	& \gr_n(G) \oplus H_2(X) \ar[r, "\widetilde{\pi}"]
		\ar[l, bend left=22, "\chi"]
	& H_2(X)\ar[r]
		\ar[l, bend left=22, "h"]
	& 0
\end{tikzcd}
\end{equation}
the map $\chi \circ \widetilde{i}$ is the identity on $\gr_n(G)$,
while $\widetilde{\pi} \circ h$ is the identity on $H_2(X)$ and
$\ker \chi = \im h$.
Every homomorphism $h\colon H_2(X) \to \gr_n(G) \oplus H_2(X)$ with
$\widetilde{\pi} \circ h$ equal to the identity on $H_2(X)$ has the form
$h(c) = (\lambda(c), c)$ and the lemma follows.
\end{proof}

In the context of the Postnikov tower \eqref{eq:postnikov} and the 
exact sequence in \eqref{equation-from-Serre-3}, this leads to a formula 
for the $k$-invariant of the extension \eqref{eq:extension}  
from $\Gg{n}$ to $\Gg{n+1}$ for a fixed $n \ge 3$, in terms a splitting map
$\sigma\colon H_2(\Gg{n}) \to \gr_n(G)$ and a map 
$h_n \colon X \to K(\Gg{n},1)$
corresponding to the projection of $G$ to $\Gg{n}$.

\begin{corollary}
\label{corollary-k-invariant}
With assumptions and notation as above, 
the $k$-invariant of the extension 
$0\to \gr_{2}(G) \to G/\Gamma_{3}(G) \to 
G/\Gamma_2(G) \to 0$
with respect to the direct sum decomposition given by the 
splitting $\sigma\colon H_2(\Gg{n}) \to \gr_n(G)$ is the element
\[
\chi_n \in \Hom(H_2(\Gg{n}), \gr_n(G)) \cong H^2(\Gg{n};\gr_n(G))
\]
given by the homomorphism 
$\chi_n (x,c) = x - \lambda (c)$, where $\lambda = \sigma \circ (h_n)_\ast 
\colon H_2(X)\to \gr_n(G)$.
\end{corollary}

\begin{proof}
The claim follows from 
Lemma \ref{lemma-bijection} and the observation that for 
the map $h$ in \eqref{equation-with-splitting-maps}, 
we have that $\lambda = \sigma \circ h$. 
\end{proof}

\begin{example}
\label{ex:free groups}
We illustrate the above corollary with a simple example (for a more 
general context, see Proposition \ref{prop:ms} below).  Let $X$ be 
a wedge of $\ell$ circles, so that $G=\pi_1(X)$ is isomorphic to $F_\ell$, 
the free group of rank $\ell$.  
Identifying $\GG{2}=\Z^\ell$ and $\gr_2(G)=\bigwedge^2 \Z^\ell$, 
the second nilpotent quotient $N=\Gg{3}$ fits into a central extension, 
\begin{equation}
\label{eq:free-extension}
\begin{tikzcd}[column sep=20pt]
0\ar[r]
	&\bigwedge^2 \Z^\ell \ar[r]
	& N \ar[r, "q_2"]
	& \Z^\ell \ar[r]
	& 0\, . 
\end{tikzcd}
\end{equation}
Note that $H_2(X)=0$, and so the homomorphism $\lambda \colon H_2(X)\to \Z^\ell$ 
is the zero map.  Hence, by Corollary \ref{corollary-k-invariant}, 
the extension \eqref{eq:free-extension} is classified by the $k$-invariant 
$\chi_2=\id \in \Hom (\bigwedge^2 \Z^\ell,\bigwedge^2 \Z^\ell)$.  
\end{example}

\section{Generalizations of Rybnikov's Theorem}
\label{sect:ryb-gen}

\subsection{The setup}
\label{subsec:setup}

Let $X$ be a connected CW-complex.  We will assume throughout 
that the homology group $H_1(X)$ is finitely generated and torsion-free, 
and that the holonomy map $(h_2)_{\ast} \colon H_2(X) \to H_1(X) \wedge H_1(X)$ 
is injective.

Let $G = \pi_1(X)$, and fix an integer $n\ge 2$.  
Recall from \eqref{eq:h2-gr} the exact sequence 
\begin{equation}
\label{eq:Stallings}
\begin{tikzcd}[column sep=24pt]
H_2(X) \ar[r, "(h_n)_\ast"] & H_2(G/\Gamma_n(G)) \ar[r]
				   & \gr_{n}(G) \ar[r]
				   & 0\, ,
\end{tikzcd}
\end{equation}
where the map $h_n\colon X \to K(G/\Gamma_n(G),1)$ 
is induced by the projection of $G\surj G/\Gamma_n(G)$. 
If $N$ is a nilpotent group with $N \cong N/\Gamma_n(N) \cong G/\Gamma_n(G)$,
then Theorem \ref{theorem-split-exact} gives a split exact sequence
\begin{equation}
\label{eq:split}
\begin{tikzcd}[column sep=18pt]
0 \ar[r] & \gr_{n}(G) \ar[r] 
			& H_2(N) \ar[r]
			& H_2(X) \ar[r]
			& 0 \, .
\end{tikzcd}
\end{equation} 

Now let $X_a$ and $X_b$ be two spaces as above 
and let $G_a$ and $G_b$ be the respective fundamental groups.  
Suppose there is a map $g\colon H^{\le 2}(X_b) \to H^{\le 2}(X_a)$ 
which is an isomorphism of graded rings. 
Set $\ov{g}\colon H_{\le 2}(X_a) \to H_{\le 2}(X_b)$ 
equal to the dual to $g$.  Then
\begin{enumerate}
\item \label{i5}
There is an isomorphism 
$G_a/\Gamma_3(G_a) \isom G_b/\Gamma_3(G_b)$.
\item \label{i6}
The isomorphism $\ov{g}_1\colon H_1(X_a) \to H_1(X_b)$ induces an
isomorphism
$\ov{g}_{\sharp}\colon \h_3(G_a) \to \h_3(G_b)$.
\end{enumerate}
If $f\colon G_a\to G_b$ is a group homomorphism, we will denote  
by $f_n\colon G_a/\G_n(G_a) \to G_b/\G_n(G_b)$ the induced 
homomorphisms between the respective nilpotent quotients.

\subsection{Statement and proof of the theorem}
\label{subsec:hom-ryb-proof}

We are ready now to state and proof our generalization 
of  \cite[Theorem 12]{Rybnikov-2}.

\begin{theorem}%[Extension of Rybnikov's Theorem]
\label{thm:ext-ryb}
With the assumptions above, fix $n\ge 3$, 
let $\sigma_{a}\colon H_2(G_a/\Gamma_n(G_a)) \to \gr_n(G_a)$
be any left splitting of the exact sequence \eqref{eq:split}, and let
$f_n\colon G_a/\Gamma_n(G_a) \to  G_b/\Gamma_n(G_b)$
be any isomorphism that extends the map $\ov{g}_1\colon 
G_a/\Gamma_2(G_a) \to  G_b/\Gamma_2(G_b)$.
The following conditions are then equivalent.
\begin{enumerate}
\item \label{rg1}
The map $\ov{g}_1$ extends to an isomorphism 
$f_{n+1}\colon G_a/\Gamma_{n+1}(G_a) \isom
G_b/\Gamma_{n+1}(G_b)$.
\item  \label{rg2}
There are liftings $h_n^{c}\colon X_c \to K(G_c/\Gamma_{n}(G_c),1)$ 
for $c = a$ and $b$ such that the following diagram commutes.
\begin{equation}
\label{eq:diagram-extended}
\begin{gathered}
\begin{tikzcd}[column sep=28pt, row sep=3.2pc]
\gr_n (G_a) \ar[r, "\ov{g}_{\sharp}", "\cong" ']
	& \gr_n(G_b)\phantom{\,.} \\
H_2(G_a/\Gamma_n (G_a)) \ar[r, "(f_n)_\ast"]
	\ar[u, "\sigma_a"]
& H_2(G_b/\Gamma_n(G_b)) \phantom{\,.}
	\ar[u, "\sigma_b" ']\\
H_2(X_a) \ar[u, "(h_n^a)_\ast"] \ar[r, "\ov{g}_2", "\cong" ']
	\ar[uu, bend left=60, "\lambda_a"]
	&	
H_2(X_b) \,.\ar[u, "(h_n^b)_\ast" ']
	\ar[uu, bend right=60, "\lambda_b" ']
\end{tikzcd}
\end{gathered}
\end{equation}
\end{enumerate}
\end{theorem}

In the above diagram, the map $\bar{g}_{\sharp}$ is the restriction 
of the map $(f_n)_{\ast}$ between the respective extensions of 
type \eqref{eq:split}. 

\begin{comment}
\begin{theorem}[Homology version of Rybnikov's Theorem]
\label{thm:h-ryb}
With the assumptions and definitions above,
let $\sigma_b\colon H_2(G_b/\Gamma_3(G_b))\to \h_3(G_b)$ 
be any left splitting of the exact sequence in 
equation \eqref{equation-from-Serre-2}, and let
$f_3\colon G_a/\Gamma_3(G_a) \isom G_b/\Gamma_3(G_b)$ 
be any extension of $\ov{g}$. Then $f_3$ extends to an isomorphism
\begin{equation}
\begin{tikzcd}[column sep=20pt]
f_4\colon G_a/\Gamma_4(G_a) \ar[r] & G_b/\Gamma_4(G_b)
\end{tikzcd}
\end{equation}
if and only if there are liftings $h_3^c\colon X_c\to K(G_c/\Gamma_3(G_c),1)$ 
for $c=a$ and $b$  such that the following diagram commutes
\begin{equation}
\label{eq:diagram}
\begin{tikzcd}[row sep=3.2pc, column sep=28pt]
\h_3(G_a) \ar[r, "\ov{g}_{\sharp}", "\cong" ']
	& \h_3(G_b)\phantom{\,.} \\
H_2(G_a/\Gamma_3(G_a)) 
       \ar[r, "(f_3)_\ast"]	
	\ar[u, "\sigma_a"]
& H_2(G_b/\Gamma_3(G_b)) \phantom{\,.}
	\ar[u, "\sigma_b" ']\\
H_2(X_a) 
        \ar[u, "(h_3^a)_\ast"] 
        \ar[r, "\ov{g}_2", "\cong" ']
	\ar[uu, bend left=64,  "\lambda_a"]
	&	
H_2(X_b) \,.
        \ar[u, "(h_3^b)_\ast" ']
	\ar[uu, bend right=64,  "\lambda_b" ']
\end{tikzcd}
\end{equation}
\end{theorem}
\end{comment}

\begin{proof}
First we show that if there is a commutative diagram 
such as the one above, then the isomorphism  
$f_n\colon G_a/\Gamma_n(G_a) \isom G_b/\Gamma_n(G_b)$  
extends to an isomorphism 
$G_a/\Gamma_{n+1}(G_a) \isom G_b/\Gamma_{n+1}(G_b)$.

From the commutativity of diagram \eqref{eq:diagram-extended}, it follows that 
$\sigma_b$ is a left splitting.  Using the direct sum decompositions 
given by the left splittings, we may define maps  
\begin{equation}
\label{eq:kappa-map}
\begin{tikzcd}
\kappa_c\colon H_2(G_c/\Gamma_n(G_c)) \cong \h_n(G_c) \oplus 
H_2(X_c) \ar[r] & \h_n(G_c)
\end{tikzcd}
\end{equation} 
for $c=a$ or $b$ by
\begin{equation}
\label{eq:kappa}
\kappa_c (x,y) = x - \lambda_c(y)\, . 
\end{equation}
Consider now the homology spectral sequences associated 
to the extensions \eqref{eq:extension} for $G=G_a$ and $G=G_b$, 
respectively.  From the naturality of the Serre spectral sequence 
and the commutativity of the aforementioned diagram, it 
follows that, with respect to the direct sum decompositions, the
map $(f_n)_\ast$ corresponds to the map 
$(x,y) \to (g_\sharp(y), \overline{g}_2(y))$.
Thus, the following diagram is commutative
\begin{equation}
\label{equation-k-invariants-2}
\begin{tikzcd}[row sep=2.8pc, column sep=28pt]
\h_n(G_a) \ar[r, "\ov{g}_\sharp"]
&\h_n(G_b) \phantom{\,.}
\\
\h_n(G_a) \oplus H_2(X_a) \ar[r, "(f_n)_\ast"] 
       \ar[u, "\kappa_a"]
	& \h_n(G_b) \oplus H_2(X_b)
	\ar[u, "\kappa_b" ']\,. 
\end{tikzcd}
\end{equation}

Let $E(\kappa_a)$ and $E(\kappa_b)$ be the central extensions with
$k$-invariants $\kappa_a$ and $\kappa_b$, respectively. Then from the 
commutativity of the diagram in \eqref{equation-k-invariants-2} it follows that
$f_n$ lifts to an isomorphism 
$\widetilde{f}_n\colon E(\kappa_a) \to E(\kappa_b)$.
On the other hand, Corollary \ref{corollary-k-invariant} implies that 
$E(\kappa_a) = K(G_a/\Gamma_{n+1}(G_a),1)$ and 
$E(\kappa_b) = K(G_b/\Gamma_{n+1}(G_b),1)$, 
and this completes the proof of the first part of the theorem.

To prove the reverse implication, assume that
a left splitting $\sigma_a\colon H_2(G_a/\Gamma_n(G_a)\to \h_n(G_a)$ 
and an isomorphism $f_n\colon G_a/\Gamma_n(G_a)\to G_b/\Gamma_n(G_b)$
are given; we will then show that there is a commutative diagram 
of the form \eqref{eq:diagram-extended}.

Let $e_{n+1}\colon G_a/\Gamma_{n+1}(G_a) \to G_b/\Gamma_{n+1}(G_b)$ 
be an isomorphism.  The first step is to prove that there is a commutative diagram
\begin{equation}
\label{equation-bottom-rectangle}
\begin{tikzcd}[row sep=3pc]
H_2(G_a/\Gamma_n(G_a)) \ar[r, "(f_n)_\ast"]	
	& H_2(G_b/\Gamma_n(G_b)) \phantom{\, .}\\
H_2(X_a)  \ar[r, "\ov{g}_2", "\cong" ']
    \ar[u, "(h_n^a)_\ast"]
	&	
H_2(X_b) \ar[u, "(h_n^b)_\ast" '] \, .
\end{tikzcd}
\end{equation}
Let $e_n$ be the isomorphism from $G_a/\Gamma_n(G_a)$ to
$G_b/\Gamma_n(G_b)$ induced by $e_{n+1}$.
Then $e_n$ gives rise to the following commutative diagram 
in the tower of nilpotent quotients
\begin{equation}
\label{eq:kggg}
\begin{tikzcd}[row sep=2.5pc]
	&K(\GGG{G_a}{n},1) \ar[r, "e_n"] \ar[d, "\pi_n^a"]
	& K(\GGG{G_b}{n},1) \ar[d, "\pi_n^b"]
	\\
	X_a \ar[r, "h_2^a"] \ar[ur, "\widetilde{h}_n^a"]
	&K(\GGG{G_a}{2},1) \ar[r, "e_2"]
	& K(\GGG{G_b}{2},1)
	& X_b \ar[l, "h_2^b" '] \ar[ul, "h_n^b" '] \, .
\end{tikzcd}
\end{equation}

Since $e_n$ and $f_n$ are both extensions of $\overline{g}$, it follows
that the automorphism 
$e_n^{-1}\circ f_n\colon \GGG{G_a}{n} \to \GGG{G_a}{n}$ is an 
extension of the identity map. This gives the following commutative
diagram
\begin{equation}
\label{eq:kggg-bis}
\begin{tikzcd}[row sep=2.5pc, column sep=1.5pc]
	&K(\GGG{G_a}{n},1) \ar[rr, "e_n^{-1}\circ f_n"]
	\ar[d, "\pi_n^a"]
	&& K(\GGG{G_a}{n},1) \ar[d, "\pi_n^a"]
	\\
	X_a \ar[r, "h_2^a" '] \ar[ur, "{h}_n^a"]
	&K(\GGG{G_a}{2},1) \ar[rr, "\id"]
	&& K(\GGG{G_a}{2},1)
	& X_a \ar[l, "h_2^a"] \ar[ul, "\widetilde{h}_n^a" '] \, .
\end{tikzcd}
\end{equation}
Putting diagrams \eqref{eq:kggg} and \eqref{eq:kggg-bis} 
together and passing to homology gives the
commutative diagram \eqref{equation-bottom-rectangle}.
Then the left splitting $\sigma_a$ determines a left splitting $\sigma_b$ such
that the upper rectangle in the diagram \eqref{eq:diagram-extended} commutes, 
and the proof of the theorem is complete.
\end{proof}

\subsection{Further refinements}
\label{subsec:refine}

The above proof shows the following: if the map 
$h_2^b\colon X_b\to K(\GGG{G_b}{2},1)$ is given, and if 
$f_n\colon G_a/\Gamma_n(G_a) \isom G_b/\Gamma_n(G_b)$ 
is an isomorphism, then there is an extension of $f_n$ 
to an isomorphism $f_{n+1}\colon \GGG{G_a}{n+1}\isom \GGG{G_b}{n+1}$
if and only if there is a lifting $h_n^a \colon X\to K(\GGG{G_a}{n},1)$ 
such that diagram \eqref{eq:diagram-extended} commutes.  
The next theorem recasts this result in a more compact fashion.

\begin{theorem} 
\label{thm:main}
With notation and assumptions as above, suppose   
$N$ is a nilpotent group with $N \cong N/\Gamma_n (N)$ and that
the map $\ell_b\colon X_b \to K(N,1)$ induces an isomorphism 
$G_b/\Gamma_n(G_b)\isom N$. 
Let $\sigma\colon H_2(N)\to \gr_n(N)$ be a splitting of the exact 
sequence \eqref{eq:split} and 
$f_n \colon \GGG{G_a}{n} \to \GGG{G_b}{n}$ an isomorphism. 
Then there is an isomorphism
\begin{equation}
\label{eq:fsub-2}
\begin{tikzcd}[column sep=20pt]
f_{n+1}\colon G_a/\Gamma_{n+1}(G_a)  \ar[r, "\cong"]
& G_b/\Gamma_{n+1}(G_b)
\end{tikzcd}
\end{equation}
extending $f_n$
if and only if there is a map $\ell_a \colon X_a \to K(N,1)$ inducing 
an isomorphism $G_a/\Gamma_n(G_a) \isom N$ such that 
the following diagram commutes:
\begin{equation}
\label{diagram-triangle-2}
\begin{tikzcd}[row sep=1.9pc]
   & \gr_n(N)\\
   {}\\
   & H_2(N) \ar[uu, "\sigma" ']\\
 H_2(X_a)\ar[rr, "\ov{g}_2", "\cong" ']
 		\ar[uuur, "\mu_a"]
 		\ar[ur, "(\ell_a)_\ast" ']
 		& & H_2(X_b)\ar[uuul, "\mu_b" ']    
		\ar[ul, "(\ell_b)_\ast"]
\end{tikzcd}
\end{equation}
\end{theorem}

\begin{proof}
Let $e_n^b$ be a isomorphism from $K(\GGG{G_b}{n},1)$ to
$K(N,1)$ and consider the following diagram
\begin{equation}
\label{eq:square-spaces}
\begin{tikzcd}[row sep=3.2pc, column sep=28pt]
K(\GGG{G_a}{n},1)
       \ar[r, "e_n^a"]	
	& K(N,1) \ar[d, "p_n"]
	& K(\GGG{G_b}{n},1) \phantom{\,.}
	\ar[l, "e_n^b" ']
	\\
X_a 
        \ar[u, "h_n^a"]
        \ar[r,  "q_a" ']
        \ar[ur, "\ell_a" ']
	&	K(\GGG{N}{2},1)
	& X_b \,. \ar[l, "q_b"] \ar[ul, "\ell_b"]
        \ar[u, "h_n^b" ']
\end{tikzcd}
\end{equation}
where $e_n^a = e_n^b \circ f_n$ and $q_a$ is determined by the condition
that on the first homology groups $\ov{q}_a = \ov{q}_b \circ \ov{g}_1$.
The corresponding diagram of homology groups and maps is
\begin{equation}
 \label{eq:square-homology}
\begin{tikzcd}[row sep=3.2pc, column sep=28pt]
H_2(\GGG{G_a}{n})
       \ar[r, "(e_n^a)_\ast", "\cong" ']	
       \ar[rr, bend left=20,  "(f_n)_\ast" ]
	& H_2(N)
	& H_2(\GGG{G_b}{n}) \phantom{\,.}
	\ar[l, "\cong" , "(e_n^b)_\ast" ']
	\\
H_2(X_a) 
        \ar[u, "(h_n^a)_\ast"]
        \ar[ur, "(\ell_a)_\ast" ']
        \ar[rr, "\cong" ', "\ov{g}_2" ]
	&	
	& X_b \,. \ar[ul, "(\ell_b)_\ast"]
        \ar[u, "(h_n^b)_\ast" ']
\end{tikzcd}
\end{equation}
 Since the maps $e_n^a$ and $e_n^b$ in  \eqref{eq:square-spaces} are
isomorphisms, it follows that there is a bijection between
 liftings $\ell_a$ and $h_n^a$ in \eqref{eq:square-spaces} and 
 also a bijection between liftings $\ell_b$ and $h_n^b$. Moreover,
 in \eqref{eq:square-homology}
 \begin{equation}
 \label{eq:equivalence}
 (\ell_a)_\ast = (\ell_b)_\ast \circ \ov{g}_2
 \quad \iff \quad
 (h_n^b)_\ast \circ \ov{g}_2 = (f_n)_\ast \circ (h_n^a)_\ast\, .
 \end{equation}
Consider now the diagram
\begin{equation}
\label{eq:big-butterfly}
\begin{tikzcd}[row sep=3.2pc, column sep=28pt]
\gr_n(G_a) \ar[r, "(e_n^a)_\sharp", "\cong" ']
		\ar[rr, bend left=20,  "\ov{g}_\sharp" ]
	& \gr_n(N) 
	& \gr_n(G_b)\phantom{\,.} 
			\ar[l, "\cong" , "(e_n^b)_\sharp" ']\\
H_2(G_a/\Gamma_n(G_a)) 
       \ar[r, "(e_n^a)_\ast", "\cong" ']
	\ar[u, "\sigma_a"]
	& H_2(N) \ar[u, "\sigma"]
	& H_2(G_b/\Gamma_n(G_b)) \phantom{\,.}
	\ar[u, "\sigma_b" ']
	\ar[l, "\cong" , "(e_n^b)_\ast" ']
	\\
H_2(X_a) 
        \ar[u, "(h_n^a)_\ast"] 
        \ar[rr,  "\ov{g}_2" ]
        \ar[ur, "(\ell_a)_\ast" ']
	\ar[uu, bend left=64,  "\lambda_a"]
	&	
	& H_2(X_b) \,.  \ar[ul, "(\ell_b)_\ast"]
        \ar[u, "(h_n^b)_\ast" ']
	\ar[uu, bend right=64,  "\lambda_b" ']
\end{tikzcd}
\end{equation} 
where the maps $(e_n^a)_\sharp$ and $(e_n^b)_\sharp$ are induced by
the corresponding isomorphisms of groups $e_n^a$ and $e_n^b$ and 
the splittings $\sigma_a$ and $\sigma_b$ are defined by requiring that
the top two rows in \eqref{eq:big-butterfly} form a commutative diagram.

From \eqref{eq:equivalence} and a diagram chase, it follows that
\eqref{eq:big-butterfly} commutes if and only if the corresponding 
diagram \eqref{eq:diagram-extended} commutes and also that 
\eqref{eq:big-butterfly} commutes if and only if diagram \eqref{diagram-triangle-2}
commutes with $\mu_a = (e_n^a)_\sharp \circ \sigma_a \circ (h_n^a)_\ast$
and $\mu_b = (e_n^b)_\sharp \circ \sigma_b \circ (h_n^b)_\ast$.
The desired conclusion follows.
\end{proof}

\begin{remark}
\label{rem:Ryb-comment}
In the work of Rybnikov \cite{Rybnikov-1,Rybnikov-2} it is assumed that 
the groups $\h_2$ and $\h_3$ are torsion-free. Then replacing the 
modules and maps in Theorem \ref{thm:ext-ryb} for $n=3$ with their 
$\Hom$ duals yields item 2 of Theorem 12 in \cite{Rybnikov-2}.
The result in Theorem \ref{thm:main} for $n=3$ corresponds to Theorem 2.2 in 
\cite{Rybnikov-1}.
\end{remark}

\subsection{The Stallings mod $p$ lower central series}
\label{subsec-p-lcs}

Let $G$ be a group, and let $p = 0$ or a prime.  Following 
Stallings \cite{Stallings}, define subgroups $\Gamma_n^p(G)<G$ 
as follows:
\begin{align*}
\Gamma_1^p(G) & = G\\
\Gamma_{n+1}^p(G) & =\langle gug^{-1}u^{-1}v^p : g \in G, 
								u, v\in \Gamma_n^p(G)\rangle\, ,
\end{align*}
where $\langle U \rangle$ denotes the subgroup generated by a 
subset $U\subset G$.  Then $\{ \Gamma_n^p(G)\}_{n\ge 1}$ is a 
descending central series of normal subgroups.  
For $p=0$ it is the lower central series; for $p\ne 0$ it is the 
most rapidly descending central series whose successive quotients
are vector spaces over the field of $p$ elements.
If we set 
$\gr_n^p(G) = \Gamma_n^p(G)/\Gamma_{n+1}^p(G)$, 
then $\gr^p(G):=\bigoplus_{n\ge 1} \gr_n^p(G)$ is a graded Lie algebra 
over $\Z_p$ in a natural way.

Now let $X$ be a path-connected space, and $G=\pi_1(X)$.
For the remainder of this section all homology groups are with
$\Z_p$ coefficients, where $\Z_0$ denotes the integers.
As shown in \cite{Stallings}, there is an exact sequence
\begin{equation}
\label{equation-Stallings-p}
\begin{tikzcd}[column sep=24pt]
H_2(X) \ar[r, "(h_n)_\ast"] & H_2(G/\Gamma_n^p(G)) \ar[r]
				   & \gr_{n}^p(G) \ar[r]
				   & 0\, ,
\end{tikzcd}
\end{equation}
where the map $h_n\colon X \to G/\Gamma_n^p(G)$ 
is induced by the projection of $G=\pi_1(X)$ to
$G/\Gamma_n^p(G)$. The proof of Theorem \ref{theorem-split-exact} 
extends to show that if $(h_2)_\ast$ is a monomorphism and $N$ is a 
nilpotent group with $N \cong N/\Gamma_n^p(N) \cong G/\Gamma_n^p(G)$,
then there is a split exact sequence
\begin{equation}
\label{equation-split-p}
\begin{tikzcd}[column sep=18pt]
0 \ar[r] & \gr_{n}^p(G) \ar[r] & H_2(N) \ar[r] & H_2(X) \ar[r] & 0
\end{tikzcd}
\end{equation} 
for all $n \ge 3$.

\subsection{An extension of Rybnikov's Theorem in characteristic $p$}
\label{subsec:ryb-mod-p}

Let $X_a$ and $X_b$ be path-connected spaces with 
$G_a = \pi_1(X_a)$ and $G_b = \pi_1(X_b)$.
Assume $p=0$ or $p$ a prime has been chosen; 
all homology groups in the following theorem are with
$\Z_p$ coefficients.  Assume also that $H_1(X_a)$ 
and $H_1 (X_b)$ are finitely generated, and 
the respective maps $(h_2)_\ast$ are monomorphisms.

Suppose we are given an isomorphism 
$g\colon H^{\le 2}(X_b) \to H^{\le 2}(X_a)$
of graded algebras. Set $\ov{g}\colon H_{\le 2}(X_a) \to H_{\le 2}(X_b)$ 
equal to the dual to $g$. 
Then given the exact sequences from \eqref{equation-Stallings-p} and 
\eqref{equation-split-p}, the steps in the proof of
Theorem \ref{thm:ext-ryb} apply to prove the following.
\begin{theorem} 
\label{thm:p-ryb}
With the assumptions above, fix $n\ge 3$, 
let $\sigma_{a}\colon H_2(G_a/\Gamma_n^p(G)) \to \gr_n^p(G_a)$
be any left splitting of the exact sequence \eqref{equation-split-p}, and let
$f_n\colon G_a/\Gamma_n^p(G_a) \to  G_b/\Gamma_n^p(G_b)$
be any isomorphism that extends the map $\ov{g}_1\colon 
G_a/\Gamma_2^p(G_a) \to  G_b/\Gamma_2^p(G_b)$.
The following conditions are then equivalent.
\begin{enumerate}
\item
The map $\ov{g}_1$ extends to an isomorphism 
$f_{n+1}\colon G_a/\Gamma_{n+1}^p(G_a) \isom
G_b/\Gamma_{n+1}^p(G_b)$.
\item There are liftings 
$h_n^{c}\colon X_c \to K(G_c/\Gamma_{n}^p(G_c),1)$ for $c = a$ 
and $b$ such that the following diagram commutes.
\begin{equation}
\label{eq:diagram-extended-p}
\begin{gathered}
\begin{tikzcd}[column sep=28pt, row sep=3.2pc]
\gr_n^p (G_a) \ar[r, "\ov{g}_{\sharp}", "\cong" ']
	& \gr_n^p(G_b)\phantom{\,.} \\
H_2(G_a/\Gamma_n^p (G_a)) \ar[r, "(f_n)_\ast"]
	\ar[u, "\sigma_a"]
& H_2(G_b/\Gamma_n^p(G_b)) \phantom{\,.}
	\ar[u, "\sigma_b"]\\
H_2(X_a) \ar[u, "(h_n^a)_\ast"] \ar[r, "\ov{g}_2", "\cong" ']
	\ar[uu, bend left=60, "\lambda_a"]
	&	
H_2(X_b) \,.\ar[u, "(h_n^b)_\ast" ']
	\ar[uu, bend right=60, "\lambda_b" ']
\end{tikzcd}
\end{gathered}
\end{equation}
\end{enumerate}
\end{theorem}

In the above diagram, the map $\bar{g}_{\sharp}$ is the restriction 
of the map $(f_n)_{\ast}$ between the respective extensions of 
type \eqref{equation-split-p}. 
The reasoning from Theorem \ref{thm:main} generalizes to show that
Theorem \ref{thm:p-ryb} implies the following result.

\begin{theorem} 
\label{thm:main-p}
With the assumptions as in Theorem \ref{thm:p-ryb}, assume 
$N$ is a nilpotent group with $N \cong N/\Gamma_n^p (N)$ and that
$\ell_b\colon X_b \to K(N,1)$ induces an isomorphism 
$G_b/\Gamma_n^p(G_b) \isom N$. 
Let $\sigma$ be a splitting of the exact sequence \eqref{equation-split-p}.
Then there is an isomorphism
\begin{equation}
\label{eq:fsub-2-p}
\begin{tikzcd}[column sep=20pt]
f_{n+1}\colon G_a/\Gamma_{n+1}^p(G_a)  \ar[r, "\cong"]
& G_b/\Gamma_{n+1}^p(G_b)
\end{tikzcd}
\end{equation}
extending $\ov{g}_2$
if and only if there is a map $\ell_a \colon X_a \to K(N,1)$ inducing 
an isomorphism of $G_a/\Gamma_n^p(G_a) \to N$ such that 
the following diagram commutes:
\begin{equation}
\label{diagram-triangle-2-bis}
\begin{gathered}
\begin{tikzcd}[row sep=1.9pc]
   & \gr_n^p(N)\\
   {}\\
   & H_2(N) \ar[uu, "\sigma" ']\\
 H_2(X_a)\ar[rr, "\ov{g}_2", "\cong" ']
 		\ar[uuur, "\mu_a"]
 		\ar[ur, "(\ell_a)_\ast" ']
 		& & H_2(X_b)\ar[uuul, "\mu_b" ']    
		\ar[ul, "(\ell_b)_\ast"]
\end{tikzcd}
\end{gathered}
\end{equation}
\end{theorem}

\section{Hyperplane arrangements}
\label{sect:hyp-arr}

We now apply the tools developed in the previous sections to a class 
of spaces that arise in a combinatorial context.  These spaces---complements  
of complex hyperplane arrangements---have motivated to 
a large extent the approach taken here, 
and provide a blueprint for further 
applications.

\subsection{Complement and intersection lattice}
\label{subsec:hyp arr}

We start with a brief review of arrangement theory; for 
details and references, we refer to the monograph of 
Orlik and Terao \cite{Orlik-Terao}.

An {\em arrangement of hyperplanes}\/ is a finite set $\A$ of 
codimension-$1$ linear subspaces in a finite-dimensional, 
complex vector space $\C^{n}$.  The combinatorics of the 
arrangement is encoded in its {\em intersection lattice}, $L(\A)$,   
that is, the poset of all intersections of hyperplanes in $\A$ (also 
known as flats), ordered by reverse inclusion, and ranked by 
codimension.  For a flat $Y=\bigcap_{H\in \B}$ defined 
by a sub-arrangement $\B\subset \A$, 
we let $\rank Y=\codim Y$; we also write 
$L_k(\A)=\{Y\in L(\A) \mid \rank Y=k\}$.

The main topological invariant associated to such an arrangement $\A$ 
is its {\em complement}, $M(\A)=\C^{n}\setminus \bigcup_{H\in\A}H$.  
This is a connected, smooth complex quasi-projective variety.  Moreover,  
$M(\A)$ is a Stein manifold, and thus has the homotopy type of a 
finite CW-complex of dimension at most $n$. 

Probably the best-known example is the braid arrangement 
$\A_n$, consisting of the diagonal hyperplanes in 
$\C^{n}$.  It is readily seen that $L(\A_{n})$ is the lattice of partitions 
of $[n]=\{1,\dots,n\}$, ordered by refinement, while   
$M(\A_{n})$ is the configuration space $F(\C,n)$ of $n$ 
ordered points in $\C$.  
In the early 1960s,  Fox, Neuwirth, and Fadell showed that 
$M(\A_n)$ is a classifying space for $P_{n}$, 
the pure braid group on $n$ strings.

For a general arrangement $\A$, the cohomology ring $H^*(M(\A),\Z)$ 
was computed by Brieskorn in the early 1970s, building on pioneering 
work of Arnol'd on the cohomology ring of the braid arrangement.  
It follows from Brieskorn's work that the space $M(\A)$ is formal over $\Q$. 
Consequently, the fundamental group of the complement, 
$G(\A)=\pi_1(M(\A),x_0)$, is $1$-formal over $\Q$.

In 1980, Orlik and Solomon gave a simple combinatorial 
description of the ring $H^*(M(\A),\Z)$:  it is the quotient 
$E(\A)/I(\A)$ of the exterior algebra $E(\A)$ on classes dual to the meridians 
around the hyperplanes, modulo a certain ideal $I(\A)$ defined 
in terms of the intersection lattice of $\A$.   
In particular, the cohomology ring of the complement is 
{\em combinatorially determined}; that is to say, if $\A$ and 
$\B$ are arrangements with $L(\A)\cong L(\B)$, then 
$H^*(M(\A),\Z)\cong H^*(M(\B),\Z)$.

\subsection{Localized sub-arrangements}
\label{subsec:local}

The {\em localization}\/ of an arrangement $\A$ at a flat $Y\in L(\A)$ 
is defined as the sub-arrangement 
\begin{equation}
\label{eq:local-arr}
\A_Y=\{H\in \A\mid H\supset Y\}\, . 
\end{equation}
The inclusion $\A_Y\subset \A$ gives rise to an inclusion of complements, 
$j_Y\colon M(\A) \inj M(\A_Y)$.  Choosing a point $x_0$ sufficiently 
close to $0\in \C^{n}$, we can make it a common
basepoint for both $M(\A)$ and all 
local complements $M(\A_Y)$.  % \cite[Lemma 4.1]{DSY-2016}, 
\begin{lemma}[\cite{DSY-2016}]
\label{lem:dsy}
There exist basepoint-preserving maps $r_Y\colon M(\A_Y)\to M(\A)$ 
such that $j_Y\circ r_Y\simeq \id$ relative to $x_0$. Moreover, if 
$H\in \A$ and $H\not\supset Y$, then the 
composite $r_Y \circ j_Y \circ r_H$ is null-homotopic.
\end{lemma}

In particular, if we set $G(\A_Y)=\pi_1(M(\A_Y),x_0)$, 
then the induced homomorphisms 
$(r_Y)_{\sharp}\colon G(\A_Y) \to G(\A)$ 
are all injective.  

The inclusions $\{j_Y\}_{Y\in L(\A)}$ assemble into a map 
\begin{equation}
\label{eq:prod-map}
\begin{tikzcd}[column sep=18pt]
j\colon M(\A) \ar[r]& \prod_{Y\in L(\A)} M(\A_Y)\, . 
\end{tikzcd}
\end{equation}
The classical Brieskorn Lemma insures that the induced 
homomorphism in cohomology is an isomorphism in each 
degree $k\ge 1$.  By the K\"unneth formula, then, we have that 
\begin{equation}
\label{eq:brieskorn}
H^k(M(\A),\Z) \cong \bigoplus_{Y\in L_k(\A)} H^k(M(\A_Y),\Z)
\end{equation}
for all $k\ge 1$.  Likewise, the Orlik--Solomon 
ideal decomposes in each degree as 
\begin{equation}
\label{eq:brieskorn-OS}
I^k(\A) \cong \bigoplus_{Y\in L_k(\A)} I^k(\A_Y)\, .
\end{equation} 
It follows that the homology groups of the complement of $\A$ 
are torsion-free, with ranks given by
\begin{equation}
\label{eq:betti-arr}
b_k(M(\A))=\sum_{Y\in L_k(\A)} (-1)^k \mu(Y)\, , 
\end{equation}
where $\mu\colon L(\A) \to \Z$ is the M\"{o}bius function 
of the intersection lattice, defined inductively by $\mu(\C^{n})=1$ 
and $\mu(Y)=-\sum_{Z\supsetneq Y} \mu(Z)$.  In particular, 
$H_1(M(\A),\Z)$ is free abelian of rank equal to the cardinality 
of the arrangement, $\abs{\A}$. 

Of particular interest to us is what happens in degree $k=2$.  
For a $2$-flat $Y$, the localized sub-arrangement $\A_Y$
is a pencil of $\abs{Y}=\mu(Y)+1$ hyperplanes.
Consequently, $M(\A_Y)$ is homeomorphic to 
$(\C \setminus \{\text{$\mu(Y)$ points}\}) \times \C^* \times \C^{n-2}$,
and so $M(\A_Y)$ is a classifying space 
for the group $G(\A_Y)=F_{\mu(Y)}\times \Z$.

\subsection{The second nilpotent quotient of an arrangement group}
\label{subsec:nilp2-arr}

Let $G=G(\A)$ be an arrangement group.  Then $G$ admits a 
commutator-relators presentation of the form $G=F/R$, where 
$F$ is the free group on generators $\{x_H\}_ {H\in \A}$ 
corresponding to meridians about the hyperplanes, and 
$R\subset [F,F]$ (see for instance \cite{Cohen-Suciu-1997} as well 
as \cite{Suciu-2001} and references therein).

Plainly, the abelianization $G_{\ab}=H_1(M(\A))$ 
is the free abelian group on $\{x_H\}_ {H\in \A}$.  
On the other hand, as noted for instance in \cite{Matei-Suciu}, 
the abelian group $\gr_2(G)$ is the $\Z$-dual of $I^2(\A)$; 
in particular, $\gr_2(G)$ is also torsion-free. 
The  central extension \eqref{eq:extension} with $n=2$ 
takes a very explicit form, detailed in the next result.

\begin{proposition}[\cite{Matei-Suciu}]
\label{prop:ms}
For any arrangement $\A$, the second nilpotent quotient of 
$G(\A)$ fits into a central extension of the form 
\begin{equation}
\label{eq:2nq-arr-ext}
\begin{tikzcd}[column sep=18pt]
0 \ar[r] & (I^2(\A))^* \ar[r] & G(\A)/\Gamma_3(G(\A))  
\ar[r] & H_1(M(\A)) \ar[r] & 0
\end{tikzcd}.
\end{equation}
Furthermore, the $k$-invariant of this extension, $\chi_2 \colon 
H_2( G_{\ab} )\to \gr_2(G)$, is the dual of the inclusion map 
$I^2(\A)\inj  E^2(\A)=\bigwedge^2 G_{\ab}$.  
\end{proposition}

It follows that $G/\Gamma_3(G)$ is the quotient of the free, 
$2$-step nilpotent group $F/\Gamma_3(F)$ by all commutation 
relations of the form 
\begin{equation}
\label{eq:nilp2-arr}
\Big[x_H , \prod_{\substack{K\in \A\\[2pt] K\supset Y}} x_{K}\Big] \, ,
\end{equation}
indexed by pairs of hyperplanes $H\in \A$ and flats $Y\in L_2(\A)$ 
such that $H\supset Y$ (see \cite{Rybnikov-1,Matei-Suciu}).  
From this description it is apparent 
that the second nilpotent quotient of an arrangement group is 
combinatorially determined.  More precisely, if $\A$ and $\B$ are 
two arrangements such that $L_{\le 2} (\A)\cong L_{\le 2} (\B)$, 
there is then an induced isomorphism,  
$G(\A)/\Gamma_3(G(\A))\cong G(\B)/\Gamma_3(G(\B))$.

\subsection{Holonomy Lie algebra}
\label{subsec:holo lie arr}

The holonomy Lie algebra of an arrangement $\A$ is 
defined as $\h(\A)=\h(G(\A))$.  Using the Orlik--Solomon description 
of the cohomology ring of $M(\A)$, it is readily seen that $\h(\A)$ 
is the quotient of $\LL(\A)$, the free Lie algebra 
on variables $\{x_H\}_{H\in \A}$, modulo the ideal  generated 
by relations arising from the rank $2$ flats:
\begin{equation}
\label{eq:holo arr}
\h(\A)=\LL( \A)\Big\slash \text{ideal}\, \Big\{
\Big[x_H , \sum_{\substack{K\in \A\\[2pt] K\supset Y}} x_{K}\Big] \:\Big| \:
\text{$H\in \A$, $Y\in L_2(\A)$, and $H\supset Y$} \Big. \Big\}  \Big. .
\end{equation}
By construction, this is a quadratic Lie algebra which depends solely 
on the ranked poset $L_{\le 2} (\A)$.  More precisely, if $\A$ and $\B$ 
are two arrangements such that $L_{\le 2} (\A)\cong L_{\le 2} (\B)$, 
there is then an induced isomorphism $\h(\A)\cong \h(\B)$.

As shown by Kohno \cite{Kohno-83} (based on foundational work by
Sullivan \cite{Sullivan} and Morgan \cite{Morgan}), the associated graded
Lie algebra $\gr(G(\A))$ and the holonomy Lie algebra $\h(\A)$
are rationally isomorphic:
\begin{equation}
\label{kohno}
 \h(\A)\otimes \Q\cong \gr(G(\A))\otimes \Q.
\end{equation}
In \cite{Falk-1988}, Falk sketched the construction of a $1$-minimal 
model for $M(\A)$ and used this to show that the rank of 
$\gr_3(G(\A))$---now sometimes known as the 
``Falk invariant" of the arrangement---is equal to the nullity  
of the multiplication map $E^1(\A)\otimes I^2(\A) \to E^3(\A)$ 
over $\Q$.  Further information on the ranks of the LCS quotients 
$\gr_n(G(\A))$ can be found in \cite{Schenck-Suciu-2002}.

At the integral level, there is a surjective
Lie algebra map, $\Psi\colon \h(\A)\surj \gr(G(\A))$, 
such that $\Psi\otimes \Q$ is an isomorphism.  In general, 
there exist arrangements for which the map $\Psi$ is not injective.  
Nevertheless, as a consequence of Theorem \ref{thm:gr3=h3}
and the preceding discussion, we have the following result. 

\begin{theorem}
\label{thm:gr3h3-arr}
For any arrangement $\A$, the map 
$\Psi_3\colon \h_3(\A)\to  \gr_3(G(\A))$ is an isomorphism.
\end{theorem}

Consequently, the group $\gr_3(G(\A))$ is combinatorially determined; 
that is, if $\A$ and $\B$ are two arrangements 
such that $L_{\le 2} (\A)\cong L_{\le 2} (\B)$, then 
$\gr_3(G(\A))\cong \gr_3(G(\B))$.  

On the other hand, 
as first noted in \cite{Suciu-2001}, there exist arrangements 
$\A$ for which $\gr_k(G(\A))$ has non-zero torsion 
for some $k>3$.  This naturally raised the question whether 
such torsion in the LCS quotients of arrangement groups is 
combinatorially determined.  The question was recently 
answered by Artal Bartolo, Guerville-Ball\'{e}, and Viu-Sos 
\cite{Bartolo}, who produced a pair of arrangements 
$\A$ and $\B$ with $L_{\le 2} (\A)\cong L_{\le 2} (\B)$, yet 
$\gr_4(G(\A))\not\cong \gr_4(G(\B))$; the difference (detected 
by computer-aided computation) lies in the $2$-torsion of the 
respective groups. 

\section{Decomposable arrangements and nilpotent quotients}
\label{sect:nilp-decomp}

We conclude with an in-depth study of a particularly nice class of 
hyperplane arrangements.  Building on work of Papadima and 
Suciu \cite{Papadima-Suciu-2006}, 
we show that the tower of nilpotent quotients of the fundamental 
group of the complement of a decomposable arrangement is fully 
determined by the intersection lattice.

\subsection{Decomposable arrangements}
\label{subsec:decomp}

Let $\A$ be an arrangement. As we saw in \S\ref{subsec:local},  
for each $2$-flat $Y\in L_2(\A)$, the group $G(\A_Y)$ is isomorphic to 
$F_{\mu(Y)}\times \Z$; hence, $\gr(G(\A_Y)) \cong \LL_{\mu(Y)} \times \LL_1$. 
Furthermore, from the defining relations \eqref{eq:holo arr}, we infer that 
$\h(\A_Y)\cong\gr(G(\A_Y))$.

Let $j$ be the map from \eqref{eq:prod-map}. Projecting onto the factors 
corresponding to rank $2$ flats we obtain a map 
\begin{equation}
\label{eq:prod-map-2}
\begin{tikzcd}[column sep=18pt]
j\colon M(\A) \ar[r]& \prod_{Y\in L_2(\A)} M(\A_Y)\, .
\end{tikzcd}
\end{equation}
The induced homomorphism on fundamental 
groups, 
\[
j_{\sharp}\colon G(\A) \to \prod_{Y\in L_2(\A)} G(\A_Y),
\] 
defines a morphism of graded Lie algebras,
\begin{equation}
\label{eq:pimap}
\begin{tikzcd}[column sep=18pt]
\h(j_{\sharp}) \colon \h(\A)\ar[r] &
\prod_{Y\in L_2(\A)} \h(\A_Y)\, .
\end{tikzcd}
\end{equation}

\begin{proposition}[\cite{Papadima-Suciu-2006}]
\label{prop:PS-holo-prod}
For any arrangement $\A$, the homomorphism 
\[
\h_n(j_{\sharp}) \colon \h_n(\A)\to \prod_{Y\in L_2(\A)} \h_n(\A_Y)
\]
is a surjection for $n\ge 3$ and an isomorphism for $n=2$.
\end{proposition}

Following \cite{Papadima-Suciu-2006}, we say that the arrangement $\A$ is {\em 
decomposable}\/ if the map $\h_3(j_{\sharp}) $ is an isomorphism (for related notions 
of decomposability, see also \cite{Cohen-Suciu-1999,Schenck-Suciu-2002}). 
The following theorem completely describes the structure of the associated 
graded and holonomy Lie algebras of a decomposable arrangement. 

\begin{theorem}[\cite{Papadima-Suciu-2006}]
\label{thm:PS-decomp}
If $\A$ is a decomposable arrangement, then the following hold:
\begin{enumerate} 
\item\label{dec1}
The map 
$\h'(j_{\sharp}) \colon \h'(\A)\to \prod_{Y\in L_2(\A)} \h'(\A_Y)$
is an isomorphism of graded Lie algebras. 
\item\label{dec2} 
The map $\Psi_{\A}\colon \h(\A)\surj \gr(G(\A))$ 
is an isomorphism.
\end{enumerate}
\end{theorem}

It follows from this theorem that the groups $\h_n(\A)=\gr_n(G(\A))$ 
are torsion-free, with ranks $\phi_n=\phi_n(G(\A))$ given by 
\begin{equation}
\label{declcs}
\prod_{n=1}^{\infty}(1-t^n)^{\phi_n}=
(1-t)^{\abs{\A}-\sum_{Y\in L_2(\A)} \mu(Y)}
\prod_{Y\in L_2(\A)} (1- \mu(Y) t)\, .
\end{equation}
Moreover, since the holonomy Lie algebra of any arrangement 
is combinatorially determined, we have the following immediate corollary.

\begin{corollary}
\label{cor:decomp-comb}
If $\A$ and $\B$ are decomposable arrangements with
$L_{\le 2}(\A) \cong L_{\le 2}(\B)$, then $\gr(G(\A))\cong \gr(G(\B))$.
\end{corollary}

\subsection{Nilpotent quotients and localized arrangements}
\label{subsec:nilp-decomp}
Our goal now is to strengthen Corollary \ref{cor:decomp-comb} 
from the level of the LCS quotients $\gr_n(G(\A))$ to the level 
of the nilpotent quotients $G(\A)/\Gamma_n(G(\A))$.  We start 
with some preparatory results on the second homology  
of these nilpotent groups.

\begin{lemma}
\label{lem:h2-decomp}
Let $\A$ be an arrangement and set $G=G(\A)$. 
There is then a natural, split exact sequence
\begin{equation}
\label{eq:h2-arr}
\begin{tikzcd}[column sep=16pt]
0\ar[r]
	&\h_{3}(\A) \ar[r]
	& H_2(G/\Gamma_3(G)) \ar[r]
	& H_2(M(\A))\ar[r]
	& 0 \, .
\end{tikzcd}
\end{equation}
Moreover, if $\A$ is decomposable, then for every $n\ge 3$ 
there is a natural, split exact sequence
\begin{equation}
\label{eq:h2-decomp}
\begin{tikzcd}[column sep=16pt]
0\ar[r]
	&\h_{n}(\A) \ar[r]
	& H_2(G/\Gamma_n(G))  \ar[r]
	& H_2(M(\A))\ar[r]
	& 0 \, .
\end{tikzcd}
\end{equation}
\end{lemma}

\begin{proof}
The first assertion follows from Theorems \ref{theorem-split-exact}  
and \ref{thm:gr3h3-arr}, while the second assertion follows from 
Theorems \ref{theorem-split-exact} and \ref{thm:PS-decomp}.
\end{proof}

For an arbitrary arrangement $\A$ and 
for a $2$-flat $Y\in L_2(\A)$, we let $\A_Y$ 
be the corresponding localized arrangement, and write 
$G_Y=G(\A_Y)$. The inclusion map $j_Y\colon M(\A)\to M(\A_Y)$ 
induces a homomorphism $(j_Y)_{\sharp}\colon G\to G_Y$ on 
fundamental groups, which in turn induces homomorphisms 
\begin{equation}
\label{eq:nilp-n-jy}
\begin{tikzcd}[column sep=18pt]
N_n(j_Y) \colon G/\Gamma_n(G)\ar[r] 
&G_Y/\Gamma_n(G_Y)\, .
\end{tikzcd}
\end{equation}
on the respective nilpotent quotients. Assembling these maps, 
we obtain a homomorphism 
\begin{equation}
\label{eq:nilp-n-map}
\begin{tikzcd}[column sep=18pt]
N_n(j) \colon G/\Gamma_n(G)\ar[r] 
&\prod_{Y\in L_2(\A)} G_Y/\Gamma_n(G_Y)\, .
\end{tikzcd}
\end{equation}

\begin{proposition}
\label{prop:nilp2-arr}
For any arrangement $\A$, and 
for each $n\ge 3$, the map $N_n(j)$   
induces a surjection in second homology, 
\begin{equation}
\begin{tikzcd}[column sep=20pt]
N_n(j)_*\colon H_2(G/\Gamma_n(G))\ar[r, two heads] 
& \bigoplus_{Y\in L_2(\A)} H_2(G_Y/\Gamma_n(G_Y))\, .
\end{tikzcd}
\end{equation}
Moreover, if $\A$ is decomposable, then the maps $N_n(j)_*$ 
are isomorphisms, for all $n\ge 3$. 
\end{proposition}

\begin{proof}
Fix $n\ge 3$, and 
set $N=G(\A)/\Gamma_n(G(\A))$ and 
$N_Y=G(\A_Y)/\Gamma_n(G(\A_Y))$. 
Consider the following diagram:
\begin{equation}
\label{eq:ladder}
\begin{tikzcd}[column sep=18pt]
0\ar[r] & \h_n(\A) \ar[r] \ar[d, "\h_n(j_{\sharp})"] 
& H_2(N) \ar[r]  \ar[d, "N_n(j)_*"] 
& H_2(M(\A))  \ar[r]  \ar[d, "j_*"]  & 0 \phantom{\, .}
\\
0\ar[r] & \bigoplus_{Y} \h_n(\A_Y) \ar[r] 
&  \bigoplus_{Y} H_2(N_Y) \ar[r] 
&\bigoplus_{Y} H_2 (M(\A_Y))  \ar[r] 
& 0 \, .
\end{tikzcd}
\end{equation}
It follows from Lemma \ref{lem:h2-decomp} that the top and bottom 
rows are (split) exact. Furthermore, the naturality 
of the exact sequence \eqref{eq:h2-decomp} implies 
that the diagram commutes.  
By Brieskorn's Lemma, the map $j_*$ is an isomorphism.  
Furthermore, by Proposition \ref{prop:PS-holo-prod}, 
the map $\h_n(j_{\sharp})$ 
is a surjection. The first claim follows at once.

If the arrangement is decomposable, then by Theorem \ref{thm:PS-decomp} 
the map $\h_n(j_{\sharp})$ is an isomorphism, whence 
$N_n(j)_*$ is also an isomorphism.
\end{proof}

\subsection{Lifting maps to nilpotent quotients}
\label{subsec:lifts}

Let $\A$ be an arrangement and set $G=\pi_1(M(\A))$. 
Composing a classifying map $M(\A)\to K(G,1)$ with the 
map $K(G,1)\to K(G_{\ab},1)$ induced by the abelianization 
homomorphism $G\to G_{\ab}$, we obtain a map of spaces, 
$h\colon M(\A)\to K(G_{\ab},1)$, uniquely defined up to homotopy. 
Fix an integer $n\ge 3$, and write $N=G/\Gamma_n(G)$ 
and $N_Y=G_Y/\Gamma_n(G_Y)$ for $Y\in L_2(\A)$. 

\begin{lemma}
\label{lem:localized-lifts}
Suppose $\ell\colon M(\A)\to K(N,1)$ is a (homotopy) lifting of $h$.  
For each $2$-flat $Y\in L_2(\A)$,  there is then a map 
$\ell_Y\colon M(\A_Y)\to K(N_Y,1)$ which lifts the map 
$h_Y\colon M(\A_Y)\to K((G_Y)_{\ab},1)$ and fits  
in the commuting diagram, 
\begin{equation}
\label{diagram:ell-maps}
\begin{tikzcd}[row sep=1.9pc, column sep=3pc]
K(N,1) \arrow[r, "N_n(j_Y)"]
	&  K(N_Y,1)  \\
M(\A) \arrow[u, "\ell"]   \arrow[r, "j_Y"]
& M(\A_Y)\arrow[u, "\ell_Y" '] \, .                          
\end{tikzcd}
\end{equation}
\end{lemma}

\begin{proof}
We define the map $\ell_Y$ by forming the composite
\begin{equation}
\label{eq:elly}
\begin{tikzcd}[column sep=26pt]
M(\A_Y) \ar[r, "r_Y"]& M(\A) \ar[r, "\ell"] & K(N,1)  \ar[r, "N_n(j_Y)"] & K(N_Y,1)\, ,
\end{tikzcd}
\end{equation}
where the first map is the one from Lemma \ref{lem:dsy}, while the 
last map is induced by the homomorphism $N_n(j_Y)\colon N\to N_Y$. 
The two claims follow at once.
\end{proof}

We now prove a converse to Lemma \ref{lem:localized-lifts}: given 
``local lifts" $\ell_Y$, there is a way to assemble them into a ``global lift," 
which we will denote by $\widetilde\ell$.  To state the result more precisely, 
start by recalling that the map $j\colon M(\A) \to \prod_{Y} M(\A_Y)$ 
induces an isomorphism 
$\h_2(j_{\sharp})\colon \h_2(G)\isom \bigoplus_Y \h_2(G_Y)$.

\begin{lemma}
\label{lem:sum}
Let $\A$ be an arrangement. Suppose that, for each 
$2$-flat $Y\in L_2(\A)$, we are given a lift 
$\ell_Y\colon M(\A_Y) \to K(N_Y,1)$ of the map 
$h_Y\colon M(\A_Y)\to K((G_Y)_{\ab},1)$. 
There is then a map $\widetilde\ell\colon M(\A) \to K(N,1)$ which lifts the map 
$h\colon M(\A)\to K(G_{\ab},1)$, and such that the following diagram commutes:
\begin{equation}
\label{diagram:coefficient-1}
\begin{tikzcd}[row sep=1.9pc]
H_2(N) \arrow[r, "N_n(j)_\ast"]
	& \bigoplus_{Y}H_2(N_Y)  \\
H_2(M(\A)) \arrow[u, "\widetilde\ell_\ast"]   \arrow[r, "j_*", "\cong" ']
& \bigoplus_{Y}H_2(M(\A_Y)) \arrow[u, "\bigoplus_Y (\ell_Y)_*" '] \, .                        
\end{tikzcd}
\end{equation}
\end{lemma}

\begin{proof}
Recall that we have a central extension 
\begin{equation}
\label{diagram:holo-nilp}
\begin{tikzcd}[row sep=18pt]
0\ar[r]& \gr_n(G)\ar[r]&  G/\Gamma_nG \ar[r]&  G/\Gamma_{n-1}G\ar[r]&  0\,.
\end{tikzcd}
\end{equation}
Recall also that the group $G$ is generated by meridians $x_H$ 
about the hyperplanes $H \in \A$, and likewise for $G_{\ab}$.  
Thus, if $\ell\colon M(\A) \to K(N,1)$ is any map  
which lifts $h\colon M(\A)\to K(G_{\ab},1)$, 
the homomorphism $\ell_{\sharp}\colon 
G \to N$ is given on generators by 
\begin{equation}
\label{eq:ell-sharp}
\ell_\sharp (x_H) = x_H a_2(H) \cdots a_{n-1} (H)\, ,
\end{equation}
for some $a_i(H) \in \gr_i(G)$.
What we need to do is pick these elements $a_i(H)$ in such a 
way so that diagram \eqref{diagram:coefficient-1} commutes.

First note the following consequence of Lemma  \ref{lem:dsy}: 
If $Z$ and $Y$ are different $2$-flats and $H \supset Z$, then
$(j_Y)_\sharp \circ (r_Z)_\sharp (x_H)$ is the identity element in
$N_Y$.

The next step is to see that if $Z$ and $Y$ are distinct $2$-flats and 
if $a \in \gr_{i}(G_Z)$, for $2 \le i \le n-1$,
then $(j_Y)_\sharp(a)$ is the identity element in $N_Y$. The group 
$\gr_{i}(G_Z)$ is generated by iterated brackets of the generators
$x_H$ for $H\in \A_Z$. 
If any one or more of these generators is replaced by the identity,
then the resulting bracket equals the identity.
Let $a$ be an iterated bracket in $\gr_i(G_Z)$. Since $a$
involves at least
one generator $x_{H^\prime}$ with $H^\prime \not\supset Y$, 
and since 
$(j_Y)_\sharp \circ (r_Z)_\sharp (x_{H^\prime})$ is the identity
in $N_Y$, it follows that 
$(j_Y)_\sharp \circ (r_Z)_\sharp (a)$ is also the identity in $N_Y$.

By \eqref{eq:ell-sharp}, the homomorphism $(\ell_Y)_\sharp\colon G_Y\to N_Y$ 
is given on generators by 
\begin{equation}
\label{eq:ell-y-sharp}
(\ell_Y)_\sharp (x_H) = x_H a_2(H,Y) \cdots a_{n-1} (H,Y)\, ,
\end{equation}
for some $a_i(H,Y) \in \h_i(G_Y)$, where $H \supset Y$.
Define a map $\widetilde{\ell} \colon M(\A) \to K(N,1)$ by requiring that
\begin{equation}
\label{diagram:ell-sharp-n}
\widetilde{\ell}_\sharp (x_H) 
= x_H \prod_{H\supset Y} a_2(H,Y) \cdots \prod_{H \supset Y}a_{n-1} (H,Y)\, .
\end{equation}

Now let $X$ and $Y$ be different $2$-flats in $L_2(\A)$ and 
consider the composition
\begin{equation}
\label{eq:composition}
\begin{tikzcd}[column sep=26pt]
\xi_{XY}\colon M(\A_X) \arrow[r,"r_X"] & M(\A) \arrow[r, "\widetilde\ell"] 
& K(N,1) \arrow[r, "N_n(j_Y)"]	& K(N_Y,1)  \, .                    
\end{tikzcd}
\end{equation}
From the result above, it follows that $(\xi_{XY})_{\sharp}(x_H)$ is 
the identity in $N_Y$ for all hyperplanes $H\in \A$ such that 
$H \supset X$ but $H \not\supset Y$.
Since there is a unique hyperplane $K$ with 
$K \supset X$ and $K \supset Y$, the image of the homomorphism 
$(\xi_{XY})_{\sharp}\colon G_X\to N_Y$ is the (infinite cyclic) 
subgroup generated by the single element 
$(\xi_{XY})_{\sharp}(x_K)$. Hence, the map $\xi_{XY}$ 
factors through $K(\Z,1)$. Since $H_2(\Z) = 0$, it follows  that 
the induced homomorphism 
$(\xi_{XY})_\ast \colon H_2(M(\A_X)) \to H_2(N_Y)$ 
is the zero map. The lemma now follows by a diagram chase.
\end{proof}

\subsection{The nilpotent quotients of a decomposable arrangement group}
\label{subsec:nilp3-decomp}

In \cite{Rybnikov-1,Rybnikov-2} Rybnikov showed that,
in general, the third nilpotent quotient of an arrangement group is {\em not}\/ 
determined by the intersection lattice.  Specifically, he produced 
a pair of arrangements $\A$ and $\B$ of $13$ hyperplanes in $\C^3$ 
such that $L (\A)\cong L (\B)$, yet 
$G(\A)/\Gamma_4(G(\A))\not\cong G(\B)/\Gamma_4(G(\B))$.
By contrast, we can use our approach to show that the phenomenon 
detected by Rybnikov cannot happen among decomposable 
arrangements. Here, then, is the main result of this section.

\begin{theorem}
\label{thm:decomposable-n}
If $\A$ and $\B$ are decomposable arrangements with
$L_{\le 2}(\A) \cong L_{\le 2}(\B)$, then, for each $n\ge 2$, 
there is an isomorphism
\begin{equation}
\label{eq:dec-iso}
\GGG{G(\A)}{n}\cong\GGG{G(\B)}{n}\, .
\end{equation}
\end{theorem}

\begin{proof}
Let $G = \pi_1(M(\B))$ and set $N=G/\Gamma_n(G)$. 
We start by picking a lifting $\ell_B\colon M(\B)\to K(N,1)$ 
of the map $M(\B)\to K(G_{\ab},1)$.  
For each $2$-flat $Z\in L_2(\B)$, 
we obtain a map $\ell_Z\colon M(\B_Z)\to K(N_Z,1)$, 
defined as in Lemma \ref{lem:localized-lifts}.  

Having an isomorphism of posets $L_{\le 2} (\A)\cong L_{\le 2} (\B)$
means we have a bijection $\A\to \B$ which induces a compatible 
bijection $L_2(\A) \to L_2(\B)$. Let  $Y\in L_2(\A)$ and $Z\in L_2(\B)$ 
be a pair of $2$-flats which correspond 
under the aforementioned bijection. Using the description of localized 
arrangement complements from \eqref{subsec:local}, we obtain 
a   $f^{YZ}\colon M(\A_Y)\to M(\B_Z)$ between the 
respective complements. 

By the forward implication of Theorem \ref{thm:main}, there is a map 
$\ell_Y \colon M(\A_Y)\to K(N_Y,1)$ and a splitting 
$\sigma_Y\colon H_2(N_Y)\to \h_n(N_Y)$ 
so that the following diagram commutes:
\begin{equation}
\label{eq:ryb-triangle-arr-n}
\begin{tikzcd}[row sep=1.5pc, column sep=1.3pc]
   & \h_n(N_Y)\\
   {}\\
   & H_2(N_Y) \ar[uu, "\sigma_Y" ']
   \\
 H_2(M(\A_Y)) \ar[rr, "f^{YZ}_*", "\cong" '] 
 \ar[uuur, "\mu_Y"]
 \ar[ur, "(\ell_Y)_\ast" ']
 & & H_2(M(\B_Z))\, .\ar[uuul, "\mu_Z" ']\ar[ul, "(\ell_Z)_\ast"]
\end{tikzcd}
\end{equation}
Using the bijection $L_2(\A) \to L_2(\B)$, the maps 
$f^{YZ}_*$ assemble to give an isomorphism 
$\Phi\colon H_2(M(\A)) \isom H_2(M(\B))$.  
The decomposability assumption together with 
Theorem \ref{thm:PS-decomp} insure that 
\begin{equation}
\label{eq:ps-dec-bis}
\h_n(G)\cong \bigoplus_{Z\in L_2(\B)} \h_n(G(\B_Z))\, .
\end{equation}
Furthermore, by Proposition \ref{prop:nilp2-arr}, we have that 
\begin{equation}
\label{eq:h2-dec-bis}
H_2(N)\cong \bigoplus_{Z\in L_2(\B)} H_2(N_Z)\, .
\end{equation}
Consequently, the homomorphisms $\sigma_Z\colon H_2(N_Y)\to \h_n(N_Y)$ 
may be assembled into a homomorphism $\sigma\colon H_2(N)\to \h_n(N)$.

\begin{comment}
Next, using the maps $\ell_Y\colon M(\A_Y)\to N_Y$, 
we define a lifting $\widetilde\ell_\A\colon M(\A) \to K(N,1)$ as 
follows.
Note that $(\ell_Y)_\sharp$ has the form
\begin{equation}
\label{diagram:ell-sharp}
(\ell_Y)_\sharp (x_H) = x_H a_2(H,X) \cdots a_{n-1} (H,X)\, ,
\end{equation}
where $H \supset Y$ and $a_i(H,X) \in \gr_i(G(\A_X))$.
Define a map $\widetilde{\ell} \colon M(\A) \to K(N,1)$ by requiring that
\begin{equation}
\label{diagram:ell-sharp-n}
\widetilde{\ell}_\sharp (x_H) 
= x_H \prod_{H\supset Y} a_2(H,Y) \cdots \prod_{H \supset Y}a_{n-1} (H,Y)\, .
\end{equation}

Recall that $(j_Y)_\sharp(x_H)$ is the identity element if $H \supsetneq Y$.
Thus, if $X$ and $Y$ are different $2$-flats in $\A$, it follows that
$(j_Y \circ \widetilde{\ell}_\A \circ r_X)_\sharp (x_H)$
is the identity element unless $H$ is the unique hyperplane with 
$H \supset X$ and $H \supset Y$. Thus, by reasoning as in the proof of 
Lemma \ref{lem:sum}, 
it follows that the following diagram commutes.
\begin{equation}
\label{diagram:coefficient-n}
\begin{tikzcd}[row sep=1.9pc]
H_2(N) \arrow[r, "N_n(j)_\ast", "\cong" ']
	& \bigoplus_{Y}H_2(N_Y)  \\
H_2(M(\A)) \arrow[u, "(\ell_{\A})_\ast"]   \arrow[r, "j_*", "\cong" ']
& \bigoplus_{Y}H_2(M(\A_Y)) \arrow[u, "\bigoplus_Y (\ell_Y)_*" '] \, ,                               
\end{tikzcd}
\end{equation}
where the isomorphism $N_n(j)_\ast$ is obtained by applying the arguments
in the proof of Proposition \ref{prop:nilp2-arr}.  
\end{comment}
Next, using the maps $\ell_Y\colon M(\A_Y)\to N_Y$, 
we define a lifting $\widetilde\ell_\A\colon M(\A) \to K(N,1)$ 
by the procedure outlined in Lemma \ref{lem:sum}.  It is 
then readily verified that the diagram
\begin{equation}
\label{eq:ryb-triangle-big-arr-n}
\begin{tikzcd}[row sep=1.5pc, column sep=1.5pc]
   & \h_n(N)\\
   {}\\
   & H_2(N) \ar[uu, "\sigma" ']
   \\
 H_2(M(\A)) \ar[rr, "\Phi", "\cong" '] 
 \ar[uuur, "\mu_{\A}"]
 \ar[ur, "(\widetilde{\ell}_{\A})_\ast" ']
 & & H_2(M(\B)) \ar[uuul, "\mu_{\B}" ']\ar[ul, "(\ell_{\B})_\ast"]
\end{tikzcd}
\end{equation}
commutes. The result follows from the backwards implication 
of Theorem~\ref{thm:main}. 
\end{proof}

%%%%%%%%%%%%%%%%%%%%%

\end{document}